\documentclass[a4paper,10pt]{article}

\usepackage{a4wide}
\usepackage[utf8]{inputenc}
\usepackage[T1]{fontenc}
\usepackage{amsmath,amssymb,commath,mathtools,amsfonts}
\usepackage{amsthm}
\usepackage{tikz}
\usepackage{pgfplots}
\pgfplotsset{compat=1.13}
\usetikzlibrary{plotmarks,calc,math,shapes}

\usepackage{enumitem}
\usepackage{booktabs}
\usepackage{subcaption}
\usepackage{todonotes}
\usepackage[colorlinks=true,citecolor=blue]{hyperref}

\newtheorem{thm}{Theorem}

\newtheorem{prop}[thm]{Proposition}

\newcounter{concounter}
\newtheorem{conprop}[concounter]{Proposition}
\renewcommand{\theconcounter}{\Alph{concounter}}

\newtheorem{lemma}[thm]{Lemma}
\newtheorem{example}[thm]{Example}
\newtheorem{remark}[thm]{Remark}

\newcommand{\R}{\mathrm{I\hspace{-0.5ex}R}}

\newcommand{\PG}{{\text{PG}}}
\newcommand{\Gas}{{\text{Gas}}}
\renewcommand{\Re}{\mathit{Re}}
\newcommand{\pu}{{\text{p.u.}}}
\newcommand{\va}{\phi} 
\newcommand{\fuel}{\varepsilon}

\makeatletter
\renewcommand{\todo}[2][]{\tikzexternaldisable\@todo[#1]{#2}\tikzexternalenable}
\makeatother

\usepgfplotslibrary{external}
\tikzsetexternalprefix{figures/}

\renewcommand{\thefootnote}{\fnsymbol{footnote}}

\begin{document}


\title{Modeling and simulation of gas networks coupled to power grids}
\author{E. Fokken\footnotemark[1], \; S. G\"ottlich\footnotemark[1], \; O. Kolb\footnotemark[1]}  

\footnotetext[1]{University of Mannheim, Department of Mathematics, 68131 Mannheim, Germany (\{fokken,goettlich,kolb\}@uni-mannheim.de).}

\date{\today}

\maketitle

\begin{abstract}
\noindent
In this paper, a mathematical framework for the coupling of gas networks to electric grids is presented to describe 
in particular the transition from gas to power.
The dynamics of the gas flow are given by the isentropic Euler equations, 
while the power flow equations are used to model the power grid. 
We derive pressure laws for the gas flow that allow for the well-posedness of the coupling and a rigorous
treatment of solutions. For simulation purposes, we apply appropriate numerical methods 
and show in a experimental study how gas-to-power might influence the dynamics of the gas and power network, respectively.
\end{abstract}



\renewcommand*{\thefootnote}{\arabic{footnote}}

\section{Introduction}
Power generation is changing drastically with more and more available energy coming from renewable sources.  Most of these, like wind and solar power, are much more volatile than traditional sources.  The problems of incorporating them into the power system have sparked much research.  Due to their flexibility in comparison to other power plants, gas turbines and gas engine power plants are often proposed as a means to balance power demands that cannot be met with renewable power sources at a given time.  The repercussions on the gas pipeline networks by this balancing have been studied~\cite{CHERTKOV2015541} and also joint optimal control of gas and power networks are investigated~\cite{ZengFangLiChen2016,Zlotnik2016}.  But so far only steady-state flow of the gas dynamics is examined.
Although this is often sufficient, the comparably low acoustic speed allows for short term effects
that cannot be seen with steady state methods.  

We propose a fully time-dependent gas model on a pipeline network coupled to a power network. Gas networks have been studied during the last two decades and questions of node conditions have been analyzed from a theoretical point of view, see for example
\cite{coupling_conditions_isothermal,banda_herty_klar,bressan_canic_garavello_herty_piccoli,brouwer_gasser_herty,doi:10.1137/060665841}.
Therein, a system of conservation or balance laws based on the Euler equations is used and the concept of Lax curves \cite{leveque2002finite} allows for a discussion on well-posedness of solutions. Finite volume or finite difference methods are typically applied to solve the network problem numerically \cite{GYRYA201934,doi:10.1137/1.9781611973693,KolbLangBales2010}.    
Power grids are mainly described by the power flow equations consisting of a nonlinear system of equations to
determine power and voltage, see for example \cite{doi:10.1137/1.9781611974164,powerflow}.

Although the power network could also be modeled with full time-dependency \cite{Gottlich2016}, the high signal speed of electricity renders this unnecessary.  Coupling of the networks happens at nodes common to both, where a gas-fired generator is placed such that it can convert gas to electricity, see \cite{aleksey_michael,Zlotnik2016}. The coupling is modeled by an equality constraint relating generated power to consumed gas flow. In section \ref{sec:gasnet} and \ref{sec:coupling-power-grids}, we derive conditions on the gas pressure and power demand, respectively, that guarantee that the coupling conditions lead to well-posedness of the gas-power coupling. For the numerical study in section \ref{sec:numerics}, we first give some validation results for the proposed discretizations, then we present the influence of different pressure laws and lastly we showcase the solution of a coupled gas-power system.


\section{Gas networks}
\label{sec:gasnet}

Gas networks have been investigated very intensively during the last ten years, see for example \cite{banda_herty_klar,brouwer_gasser_herty,KolbLangBales2010}. Coupling conditions at nodes
have been established to ensure well-posedness of the network solution~\cite{doi:10.1137/060665841} and a rigorous numerical treatment \cite{Egger2018,GYRYA201934,KolbLangBales2010,zhou_adewumi}. 
As underlying structure for gas networks we consider a directed graph with nodes $V_\Gas$ and arcs $E_\Gas$. 
Nodes as well as arcs may have state variables, which are in general time-dependent. We briefly recall the common setting.

In each pipe $e\in E_\Gas$, the gas flow is described by the isentropic Euler equations in one space dimension:
\begin{equation}\label{eq:gas}
    \begin{pmatrix}\rho\\  q\end{pmatrix}_t + \begin{pmatrix} q\\ p(\rho) + \frac{ q^2}{\rho}\end{pmatrix}_x = \begin{pmatrix}0\\S(\rho, q) \end{pmatrix},
\end{equation}
where $t \in \R^+$ is the time, $x \in [0,l_e]$ is the position along the pipeline of length~$l_e$, $\rho$ is the density of the gas, $q$ is the momentum of the gas, $p$ is the pressure and $S$ includes source terms that influence the momentum of the gas.
The pressure law must be a function of the form
\begin{equation*}
  \begin{aligned}
    p \in C^1(\R^+, \R^+),\quad
    p^\prime(\rho) > 0 \text{ for all }\rho \in \R^+,
\end{aligned}
\end{equation*}
to ensure strict hyperbolicity.  Usual examples are the isothermal pressure law
\begin{equation}
p(\rho) = c^2\rho ,
\end{equation}
where $c$ is the acoustic speed of the gas, or, more general, the $\gamma$-law
\begin{equation}\label{eq:gammalaw}
p(\rho) = \kappa \rho^\gamma, 
\end{equation}
for suitable constants $\kappa$ and $\gamma$, which we examine in Proposition \ref{prop:general_gamma}. 
In fact, we show that the class of possible pressure laws can be enlarged, leading to non-standard pressure functions.
%
%
From literature we know that the eigenvalues $\lambda$ and the eigenvectors $r$ for the system~\eqref{eq:gas}
are given by
\begin{align*}
    \lambda_1(\rho, q) = \frac{q}{\rho} - \sqrt{p^\prime(\rho)}, \quad \lambda_2(\rho, q) = \frac{q}{\rho} + \sqrt{p^\prime(\rho)}, \\[10pt]
  r_1(\rho, q) = \begin{pmatrix} -1 \\ -\lambda_1(\rho, q) \end{pmatrix}, \quad r_2(\rho, q) = \begin{pmatrix} 1 \\ \lambda_2(\rho, q) \end{pmatrix}\ .
\end{align*}
At the (gas) network inflow boundaries, typically the pressure $p_{\text{in}}$ is prescribed and at the outflow boundaries the flow $q_{\text{out}}$ is given.
Further, the following coupling conditions, an introduction to which can be found in \cite{banda_herty_klar}, are used at all inner nodes of $V_\Gas$:
\begin{itemize}
\item equality of pressure: the pressure values at the ends of all arcs connected to the same node must be equal, that is, there is a coupling pressure $p_\text{coupling}$ such that
  \begin{equation}
    \label{eq:pressure_coupling}
    p_e = p_\text{coupling}
  \end{equation}
  at the end of all arcs $e$ connected to the junction.  This condition is equivalent to a condition
    \begin{equation*}
    \rho_e = \rho_\text{coupling}\ ,
  \end{equation*}
because $p$ as a strictly increasing function is one-to-one.
\item conservation of mass: the sum of all incoming fluxes must equal the sum of all outgoing fluxes (including eventual source/sink terms), i.e.,
  \begin{equation}
    \label{eq:flux_conservation}
    \sum_{\text{incoming pipes}}q_{\text{pipe}}=\sum_{\text{outgoing pipes}}q_{\text{pipe}}
  \end{equation}
\end{itemize}

\subsection{Well-posedness of the Riemann problem for the isentropic Euler equation}\label{sec:analysisOutflowBounds}
In this section, we investigate conditions for the pressure function~\eqref{eq:gammalaw} and the flows that guarantee well-posedness of the Riemann problem for the isentropic Euler equations without a source term, see \cite{bressan_canic_garavello_herty_piccoli} for an introduction.  By the front-tracking technique this is sufficient for the well-posedness of the isentropic Euler equations, see \cite{gugat_herty_klar_et_al}.
Assuming strict hyperbolicity, we start with the assumption $p^\prime>0$.  We also assume that the initial data $U_l = (\rho_l,q_l)$ and $U_r = (\rho_r,q_r)$ of the Riemann problem is sub-sonic, i.e.,
\begin{equation}
  \label{eq:39}
  \abs{\frac{q}{\rho}} < c(\rho) = \sqrt{p^\prime(\rho)}.
\end{equation}
A solution to the Riemann problem with left state $U_l$ and right state $U_r$ is given by an intersection point of the Lax curves through these two points, again see \cite{bressan_canic_garavello_herty_piccoli}.
The Lax curve of all states $U=(\rho,q)$ reachable via $1$-rarefaction or $1$-shock (that is, waves corresponding to $\lambda_1$ above) from a left state $U_l = (\rho_l,q_l)$ is given by (taken from \cite{doi:10.1137/060665841})
\begin{equation}\label{eq:L_l}
  L_l(\rho;\rho_l,q_l) =\begin{cases}
    \rho\left(\frac{q_l}{\rho_l} +\int_{\rho}^{\rho_l} \frac{c(s)}{s} \dif s \right) &\text{ for }\rho\phantom{_l} \leq \rho_l \text{ (rarefaction)}\\[2ex]
    \rho \frac{q_l}{\rho_l} - \sqrt{f(\rho,\rho_l)}\ &\text{ for }\rho_l \leq \rho\phantom{_l} \text{ (shock)},
  \end{cases}
\end{equation}
where $f$ is defined by
\begin{equation}\label{eq:43}
  \begin{aligned}
    f(\rho,\rho_l) = a(\rho,\rho_l)\Delta p(\rho,\rho_l) \quad \mbox{with }\;
    a(\rho,\rho_l)= \frac{\rho}{\rho_l}(\rho-\rho_l), \;
    \Delta p(\rho,\rho_l) = p(\rho)-p(\rho_l)\ 
  \end{aligned}
\end{equation}
such that $U = (\rho,L_l(\rho))$ holds.

Analogously, the Lax curve of all states $U=(\rho,q)$ reachable via $2$-rarefaction or $2$-shock from a right state $U_r = (\rho_r,q_r)$ is given by
\begin{equation}\label{eq:L_r}
  L_r(\rho;\rho_r,q_r) =\begin{cases}
    \rho\left(\frac{q_r}{\rho_r} - \int_{\rho}^{\rho_r} \frac{c(s)}{s} \dif s \right) &\text{ for }\rho\phantom{_r} \leq \rho_r\text{ (rarefaction)}\\[2ex]
    \rho \frac{q_r}{\rho_r} + \sqrt{f(\rho,\rho_r)}\ &\text{ for }\rho_r \leq \rho\phantom{_r}\text{ (shock)}.
  \end{cases}
\end{equation}
Whenever we write $L_l(\rho)$ and $L_r(\rho)$, it will mean $L_l(\rho,\rho_l,q_l)$ and $L_r(\rho,\rho_r,q_r)$, respectively.
Beware that later in section \ref{sec:extension-junctions} different second and third arguments will appear and we will write them out again.

A solution is then at a point $\rho$, where
\begin{equation}
  \label{eq:10}
  L_l(\rho)-L_r(\rho) = 0\ .
\end{equation}
An example of this is shown in figure \ref{fig:lax_curves}.
\begin{figure}[h]
  \centering
  \includegraphics{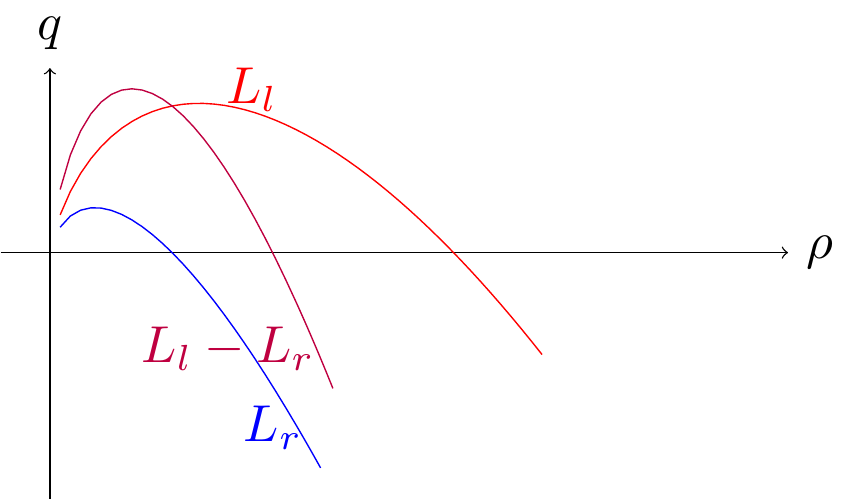}
  \caption{Typical Lax curves with pressure law $p(\rho) = c^2\rho$.  The zero of the purple curve is the desired solution.}
  \label{fig:lax_curves}
\end{figure}

Note that
\begin{equation}
  \label{eq:9}
  - L_r(\rho;\rho_r,q_r) = L_l(\rho;\rho_r,-q_r)\ .
\end{equation}
So $L_l$ and $-L_r$ share most properties.  Therefore equation \eqref{eq:10} is best understood not as a difference of two functions but as a sum of two very similar functions.
Similar to \cite{aleksey_michael}, we show certain properties of the Lax curves
that are used to define the new pressure laws.  There will appear three propositions, which are numbered by A, B and C. These propositions present lists of conditions that are closely related.  To make things tractable, each list bears the same numbering but has the letter of the corresponding proposition in front.  For example, the conditions A1, B1 and C1 all govern concavity of the Lax curves.

 \stepcounter{thm}
 \begin{conprop}[Proposition \thethm]
\label{prop:Lax_curves}
  Let $L\in C^2(\mathbb{R}^+, \mathbb{R}^+)$ and let it fulfill conditions \ref{item:concave} through \ref{item:positive} (where only one of \ref{item:a_positive} and \ref{item:b_positive} must hold).

  \begin{enumerate}[label=\normalfont{(\theconcounter\arabic*)},ref=\normalfont{\theconcounter}\arabic*]
  \item $L^{\prime\prime}\leq 0$,\label{item:concave}
  \item $\exists \rho>0: L(\rho) <0$,\label{item:negative}
  \item \label{item:positive}\begin{enumerate}[label=(\alph*), ref=\normalfont{\theconcounter}\ensuremath{\arabic{enumi}}(\alph{enumii})]
    \item $0 <\lim_{\rho \to 0}L(\rho) <\infty$ or \label{item:a_positive}
    \item $\lim_{\rho \to 0} L(\rho) = 0$ and $\lim_{\rho \to 0}L^\prime(\rho) > 0$.\label{item:b_positive}
    \end{enumerate}
  \end{enumerate}
  Then there is a unique $\rho>0$ with $L(\rho)=0$ such that for all $\hat{\rho}>\rho$ holds $L(\hat{\rho})<0$ and $L(\rho)\to -\infty$ for $\rho \to \infty$.
\begin{proof}
  Condition \ref{item:positive} ensures positive values near $0$, condition \ref{item:negative} ensures negative values for some $\rho>0$, together yielding a zero in between and condition \ref{item:concave} makes $L$ concave, guaranteeing uniqueness of the zero and implying then that $L(\rho) \to -\infty$ for $\rho \to \infty$.
\end{proof}

\end{conprop}

Note if two functions $L_1$, $L_2$ satisfy the prerequisites of Proposition \ref{prop:Lax_curves}, so does $L_1 + L_2$.  Therefore we search for conditions on pressure functions that make $L_l$ and $-L_r$ satisfy the conditions of Proposition \ref{prop:Lax_curves}, as then their sum in equation \eqref{eq:10} will, too, and so it will have a unique zero.  

Our main result is the following.  It generalizes the usual notion of $\gamma$-laws to allow for $\gamma<1$. Further, as we will show below (Proposition~\ref{prop:sufficient}), positive linear combinations of such valid pressure laws again lead to valid pressure laws. 
\begin{prop}[generalized $\gamma$-laws]\ \label{prop:general_gamma}
  Let the derivative of the pressure be given by
  \begin{equation*}
    p^\prime(\rho) = \alpha \rho^{\delta}.
  \end{equation*}
  This translates to pressure functions given by
  \begin{equation*}
    p(\rho) = \frac{\alpha}{\gamma}\rho^\gamma+\text{const}
  \end{equation*}
  for $\gamma \ne 0$ and
  \begin{equation*}
    p(\rho) = \alpha\log(\rho)+\text{const}\ .
  \end{equation*}
  For these pressure functions there holds
  \begin{itemize}
  \item The Lax curves through any sub-sonic initial states $U_l$, $U_r$ have a unique intersection point at some $\rho>0$ if and only if $\abs{\delta}\leq 2$ and $\alpha >0$.
  \item For every $\delta$ with $\abs{\delta}>2$ there are sub-sonic states $U_l$, $U_r$ such that the Lax curves have no intersection at all.
  \end{itemize}
\end{prop}
Expressed in the usual form of $\gamma$-laws, this result means $p(\rho) = \kappa \rho^{\gamma}$ is a valid pressure function if and only if $0<\gamma<3$ and $\kappa>0$ or $-1<\gamma<0$ and $\kappa<0$. The proof of Proposition \ref{prop:general_gamma} will be given at the end of this section.
To prove Proposition \ref{prop:general_gamma} we will reformulate conditions \ref{item:concave} through \ref{item:positive} in Proposition \ref{prop:Lax_curves} in terms of the pressure function.

Because of equation \eqref{eq:9}, both Lax curves behave essentially identically and therefore we prove our findings for $L_l$ only as the proofs for $-L_r$ are the same.
Along the way the derivatives of the Lax curves will be important and so we provide them here,
\begin{equation}\label{eq:derivative_L_l}
  \begin{aligned}
  L_l^\prime(\rho) &=\begin{cases}
    \frac{q_l}{\rho_l}+ \int_{\rho}^{\rho_l}  \frac{c(s)}{s}\dif s - c(\rho)&\text{ for }\rho\phantom{_l} \leq \rho_l\\[2ex]
    \frac{q_l}{\rho_l} - \frac{f^\prime}{2\sqrt{f}}\ &\text{ for }\rho_l \leq \rho
  \end{cases}\\
  L_l^{\prime\prime}(\rho) &=\begin{cases}
     - (\frac{c(\rho)}{\rho} + c^\prime(\rho))\qquad\ \ \, \, &\text{ for }\rho\phantom{_l} \leq \rho_l\\[2ex]
     - \frac{2f^{\prime\prime}f- (f^\prime)^2}{4\sqrt{f}}&\text{ for }\rho_l \leq \rho\ ,
  \end{cases}\\
  \end{aligned}
\end{equation}
and
\begin{equation}\label{eq:derivative_L_r}
  \begin{aligned}
  L_r^\prime(\rho) &=\begin{cases}
  \frac{q_r}{\rho_r} - \int_{\rho}^{\rho_r}  \frac{c(s)}{s}\dif s +c(\rho)&\text{ for }\rho\phantom{_r} \leq \rho_r\\
     \frac{q_r}{\rho_r} + \frac{f^\prime}{2\sqrt{f}}\ &\text{ for }\rho_r \leq \rho
  \end{cases}\\
  L_r^{\prime\prime}(\rho) &=\begin{cases}
     \frac{c(\rho)}{\rho} + c^\prime(\rho)\qquad\qquad\ \,\, \, &\text{ for }\rho\phantom{_r} \leq \rho_r\\
     \frac{2f^{\prime\prime}f- (f^\prime)^2}{4\sqrt{f}}\ &\text{ for }\rho_r \leq \rho\ .
  \end{cases}
\end{aligned}
\end{equation}
Next, we intend to give conditions under which Proposition \ref{prop:Lax_curves} is applicable to $L_l$.  
Before we do so, we provide the following lemma:
\begin{lemma}
  \label{sec:condition-lemma}
  Let $g\in C^1(\mathbb{R}^+, \mathbb{R}^+)$ be a non-negative function, $g\geq0$, such that also $(\rho^2g(\rho))^\prime \geq 0$ for all $\rho>0$ and let $G$ be given by $G(\rho) = \int_\rho^{\rho_l}g(s)\dif s$.  Then, there holds
  $\rho^2g(\rho) \overset{\rho \to 0}{\longrightarrow} 0$ if and only if $\rho G(\rho) \overset{\rho \to 0}{\longrightarrow} 0$.
  \begin{proof}
    The proof can be found in Appendix \ref{sec:proofs-lemm-refs}.
  \end{proof}
\end{lemma}
The role of $g$ in Lemma \ref{sec:condition-lemma} will be played by $p^\prime(\rho)$ and $\frac{c(\rho)}{\rho}$ in the following.
Let us now put the focus on the Lax curve $L_l$ again.
The following Proposition \ref{prop:proposition_b} will turn conditions \ref{item:concave} through \ref{item:positive} in Proposition \ref{prop:Lax_curves} into conditions on the pressure function.
\stepcounter{thm}
\begin{conprop}[Proposition \thethm]
  \label{prop:proposition_b}
  Let $p\in C^2(\mathbb{R}^+)$ with $p^{\prime}>0$.  Then conditions \ref{item:concave}, \ref{item:negative}, \ref{item:positive} hold for $L_l$ for all $\rho_l$, $q_l$ with $\abs{\frac{q_l}{\rho_l}} <  \sqrt{p^\prime(\rho_l)}$ if and only if conditions \ref{item:B_concave} through \ref{item:B_positive} hold (again with only one of \ref{item:B_a_positive} and \ref{item:B_b_positive} fulfilled).
	
  \begin{enumerate}[label=\normalfont{(\theconcounter\arabic*)},ref=\normalfont{\theconcounter}\arabic*]
  \item These inequalities hold:\label{item:B_concave} 
    \begin{align}
        2p^\prime(\rho)+\rho p^{\prime\prime}(\rho)&\geq 0\ \forall \rho>0\label{eq:46}\\
        \Delta p^2 + a^2 \left( 2\Delta pp^{\prime\prime}-(p^\prime)^2 \right) +  \frac{1}{2}(a^2)^{\prime}(\Delta p^2)^{\prime}& \geq 0\ \forall \rho_l>0,\rho>\rho_l,\label{eq:47}
      \end{align}
        where the arguments $\rho$ and $\rho_l$ have been omitted for readability, see equation \eqref{eq:43}.
       \item \label{item:B_negative}Let $p_\infty = \lim_{\rho \to \infty} p(\rho) \in \mathbb{R} \cup \{\infty\}$.  For all $\rho >0$ there holds
       \begin{equation*}
         p_\infty  -\rho p^\prime(\rho) - p(\rho) \geq 0\ .
       \end{equation*} 
       \item\label{item:B_positive}
         \begin{enumerate}[label=(\alph*), ref=\normalfont{\theconcounter}\ensuremath{\arabic{enumi}}(\alph{enumii})]
    \item There is $p_0 >0$ such that $p(\rho) = -\frac{p_0}{\rho} + \underset{\rho \to 0}{o}\left( \frac{1}{\rho} \right)$ or \label{item:B_a_positive}
    \item $p(\rho) \in \underset{\rho \to 0}{o}\left( \frac{1}{\rho} \right)$ and $\lim_{\rho \to 0}\int_{\rho}^{\rho_l}  \frac{c(s)}{s}\dif s - c(\rho) - c(\rho_l) \geq 0$
      for all $\rho_l>0$.\label{item:B_b_positive}
    \end{enumerate}

     \end{enumerate}
  
     \begin{proof}\ 
       \begin{itemize}
       \item[\ref{item:B_concave} $\Leftrightarrow$ \ref{item:concave}:] The equivalence \ref{item:B_concave} $\Leftrightarrow$ \ref{item:concave} is immediate from the definitions, inequality \eqref{eq:46} is for the rarefaction part, inequality \eqref{eq:47} for the shock part.  Note that
         \begin{equation*}
            2f f^{\prime\prime}-(f^\prime)^2 = \Delta p^2 + a^2 \left( 2\Delta pp^{\prime\prime}-(p^\prime)^2 \right) +  \frac{1}{2}(a^2)^{\prime}(\Delta p^2)^{\prime}. 
         \end{equation*}
       \item[\ref{item:concave}, \ref{item:a_positive}$\Rightarrow$ \ref{item:B_concave},\ref{item:B_a_positive}:]
        Assume conditions \ref{item:concave} and \ref{item:a_positive}. Let $F(\rho) =
           \int_\rho^{\rho_l}\frac{c(s)}{s}\dif s$.
         \begin{equation}
           \label{eq:44}
           0<l = \lim_{\rho \to 0} L_l(\rho) =  \lim_{\rho \to 0} \rho \int_\rho^{\rho_l}\frac{c(s)}{s}\dif s =  \lim_{\rho \to 0} \rho F(\rho).
         \end{equation}
         Therefore $F(\rho) = \frac{l}{\rho} + \underset{\rho \to 0}{o}\left( \frac{1}{\rho} \right)$.
         Note now that
         \begin{equation*}
           \left( \rho^2 \frac{c(\rho)}{\rho} \right)^\prime = \frac{1}{2\sqrt{p^\prime(\rho)}}(2p^\prime(\rho)+ \rho p^{\prime\prime}(\rho))\geq 0
         \end{equation*}
         due to condition  \ref{item:concave}.  Lemma \ref{sec:condition-lemma} therefore shows that
         \begin{equation}
           \label{eq:48}
           \frac{c(\rho)}{\rho} = - F^{\prime}(\rho) = \frac{l}{\rho^2}+ \underset{\rho \to 0}{o}\left( \frac{1}{\rho^2} \right)
         \end{equation}
         and hence
         \begin{equation*}
           p^{\prime}(\rho) = c(\rho)^2= \frac{l^2}{\rho^2} + \underset{\rho \to 0}{o}\left( \frac{1}{\rho^2} \right)\ ,
         \end{equation*}
         which again with lemma \ref{sec:condition-lemma} yields
         \begin{equation*}
         p(\rho) = -\frac{l^2}{\rho} + \underset{\rho \to 0}{o}\left( \frac{1}{\rho} \right)\ .
       \end{equation*}
     \item[\ref{item:concave}, \ref{item:a_positive}$\Leftarrow$ \ref{item:B_concave},\ref{item:B_a_positive}:] Assume now conditions \ref{item:B_concave} and \ref{item:B_a_positive}.  We now note that $\left( \rho^2 p^\prime(\rho) \right)^\prime = \rho\left( 2p^\prime(\rho)+ \rho p^{\prime\prime}(\rho) \right)\geq 0$, because of \ref{item:B_concave} and use the lemmas to  arrive at \ref{item:a_positive}.
     \item[\ref{item:concave}, \ref{item:b_positive}$\Leftarrow$ \ref{item:B_concave},\ref{item:B_b_positive}:] We assume condition \ref{item:B_b_positive}.  The last proof (here we need again conditions \ref{item:concave} and \ref{item:B_concave} respectively) also shows that
       \begin{equation*}
         \lim_{\rho \to 0}L_l(\rho) = 0 \Leftrightarrow p \in \underset{\rho \to 0}{o}\left( \frac{1}{\rho} \right)\ .
       \end{equation*}
       For the derivative we find
       \begin{equation*}
         \lim_{\rho \to 0}L_l^\prime(\rho) = \lim_{\rho \to 0}\frac{q_l}{\rho_l}+ \int_{\rho}^{\rho_l}  \frac{c(s)}{s}\dif s - c(\rho)> \lim_{\rho \to 0}\int_{\rho}^{\rho_l}  \frac{c(s)}{s}\dif s - c(\rho) -c(\rho_l)
       \end{equation*}
       due to the sub-sonic condition.  This is non-negative due to condition \ref{item:B_b_positive}.
     \item[\ref{item:concave}, \ref{item:b_positive}$\Rightarrow$ \ref{item:B_concave},\ref{item:B_b_positive}:] Assume that \ref{item:B_b_positive} does not hold, that is
       \begin{equation*}
         \lim_{\rho \to 0}\int_{\rho}^{\rho_l}  \frac{c(s)}{s}\dif s - c(\rho) -c(\rho_l)\leq -\delta<0\ .
       \end{equation*}
       Choosing $\rho_l,q_l$ such that $\frac{q_l}{\rho_l} = -c(\rho_l)+\frac{1}{2}\delta$ yields
       \begin{equation*}
         \lim_{\rho \to 0}L_l^\prime(\rho)\leq -\frac{1}{2}\delta<0
       \end{equation*}
     \item[\ref{item:negative}$\Leftarrow$\ref{item:B_negative}:] First of all we note that the limit exists, as $p^\prime >0$.  For $\rho>\rho_l$ we find
       \begin{equation*}
       L_l(\rho) = \rho\left( \frac{q_l}{\rho_l} - \sqrt{\frac{f(\rho,\rho_l)}{\rho^2}} \right) = \rho\left( \frac{q_l}{\rho_l} - \sqrt{\frac{1}{\rho_l}\left( 1-\frac{\rho_l}{\rho} \right)\left( p(\rho)-p(\rho_l) \right)} \right)\ .
     \end{equation*}
     This is less than zero if and only if the bracket is less then zero.  If now $p_\infty = \infty$, this is fulfilled for some $\rho>\rho_l$.  If $p_\infty< \infty$, we have the limit
     \begin{equation*}
       \lim_{\rho \to \infty}\sqrt{\frac{1}{\rho_l}\left( 1-\frac{\rho_l}{\rho} \right)\left( p(\rho)-p(\rho_l) \right)} = \sqrt{\frac{p_\infty - p(\rho_l)}{\rho_l}}\ .
     \end{equation*}
     Let now $\delta$ and $\epsilon_\rho$ be defined by
     \begin{equation*}
       \begin{aligned}
         \frac{q_l}{\rho_l} + \delta &= c(\rho_l)\\
       \sqrt{\frac{1}{\rho_l}\left( 1-\frac{\rho_l}{\rho} \right)\left( p(\rho)-p(\rho_l) \right)} &= \sqrt{\frac{p_\infty - p(\rho_l)}{\rho_l}} - \epsilon_\rho\ .
     \end{aligned}
     \end{equation*}
     Then $\delta >0$ as the state is sub-sonic and $\epsilon_\rho \to 0$.  Therefore let $\rho$ be so great that $\delta - \epsilon_\rho>0$.  Then, we find
     \begin{equation*}
       \begin{aligned}
         \frac{q_l}{\rho_l} &<   \frac{q_l}{\rho_l} + \delta -\epsilon_\rho= c(\rho_l) -\epsilon_\rho
         \leq \sqrt{\frac{p_\infty - p(\rho_l)}{\rho_l}}  -\epsilon_\rho \\
         & = \sqrt{\frac{1}{\rho_l}\left( 1-\frac{\rho_l}{\rho} \right)\left( p(\rho)-p(\rho_l) \right)} 
       \end{aligned}
     \end{equation*}
     which implies \ref{item:negative}.
   \item[\ref{item:negative}$\Rightarrow$\ref{item:B_negative}:]  In case $p_\infty = \infty$ condition \ref{item:B_negative} is fulfilled.  Let now $p_\infty<\infty$ and let there be for all sub-sonic $(\rho_l,q_l)$ a $\rho_->0$ such that
     \begin{equation}
       \label{eq:2}
       \frac{q_l}{\rho_l} < \sqrt{\frac{1}{\rho_l}\left( 1-\frac{\rho_l}{\rho_-} \right)\left( p(\rho_-)-p(\rho_l) \right)}\ ,
     \end{equation}
     that is, let \ref{item:negative} be true. Note that the right-hand-side is increasing in $\rho_-$ because its derivative is positive.  Therefore taking suprema of \eqref{eq:2} yields
       \begin{equation*}
         c(\rho_l)  \leq \sqrt{\frac{p_\infty - p(\rho_l)}{\rho_l}} \quad \Rightarrow \quad p^\prime(\rho_l)  \leq \frac{p_\infty - p(\rho_l)}{\rho_l}.  \qedhere
       \end{equation*}
     \end{itemize}      
     \end{proof}
\end{conprop}
With Proposition \ref{prop:proposition_b} we have a list of conditions \ref{item:B_concave},~\ref{item:B_negative},~\ref{item:B_positive} which is equivalent to conditions \ref{item:concave},~\ref{item:negative},~\ref{item:positive} from Proposition \ref{prop:Lax_curves} but now only involves the pressure function.
This is fortunate as we want to prove Proposition \ref{prop:general_gamma}, which only contains pressure functions.
In principle, we could just plug the pressure functions into our conditions \ref{prop:proposition_b} and check, what generalized $\gamma$-laws are allowed.
But there is more insight to be gained by simplifying the conditions further.
We will do so in Propositions \ref{prop:B1_simpler},~\ref{prop:B3_simpler} and \ref{prop:B2_simpler}.
Yet this comes at a price, the next list of conditions in Proposition \ref{prop:sufficient} below will only be sufficient conditions for Propositions \ref{prop:Lax_curves} and \ref{prop:proposition_b} to hold.
\begin{prop}
  \label{prop:B1_simpler}
  Let $p \in C^3(\mathbb{R}^+)$.  Condition \ref{item:B_concave} holds if
    \begin{equation*}
      \begin{aligned}
        2p^\prime(\rho)+ \rho p^{\prime\prime}(\rho)&\geq 0\\
        6p^\prime(\rho) + 6\rho p^{\prime\prime}(\rho)+\rho^2p^{\prime\prime\prime}(\rho)&\geq 0
      \end{aligned}
    \end{equation*}
    holds for all $\rho > 0$.
    \begin{proof}
      We note that $\left( 2f f^{\prime\prime}-(f^\prime)^2  \right)(\rho_l,\rho_l) = 0$ and that
      \begin{equation*}
        \left( 2f f^{\prime\prime}-(f^\prime)^2  \right)^\prime = 2f f^{\prime\prime\prime}\ .
      \end{equation*}
      As $f\geq 0$, we see that $ 2f f^{\prime\prime}-(f^\prime)^2 \geq 0$ for all $\rho\geq\rho_l$ if $f^{\prime\prime\prime}(\rho,\rho_l)\geq 0$ for all $\rho\geq\rho_l$.  This was also proved in \cite{MartinHertyMueller2017}.  This is the case if and only if $\rho_l f^{\prime\prime\prime}(\rho,\rho_l)\geq 0$.  For this we find
    \begin{equation*}
      \begin{aligned}
       \rho_l f^{\prime\prime\prime}(\rho,\rho_l) &= 6p^\prime(\rho) + 3(2\rho-\rho_l)p^{\prime\prime}(\rho)+ \rho(\rho-\rho_l)p^{\prime\prime\prime}(\rho)\\
        &= \left[ 6p^\prime(\rho) + 6\rho p^{\prime\prime}(\rho)+\rho^2p^{\prime\prime\prime}(\rho)\right] -\rho_l \left[ 3p^{\prime\prime}(\rho)+ \rho p^{\prime\prime\prime}(\rho) \right], 
      \end{aligned}
    \end{equation*}
    which is an affine function in $\rho_l$.  Hence for given $\rho>0$ this function takes its minimum in one of the edges of the simplex $\{\rho_l \ | \ 0\leq \rho_l\leq \rho\}$.  The values on these are given by
      \begin{align*}
        &6p^\prime(\rho) + 6\rho p^{\prime\prime}(\rho)+\rho^2p^{\prime\prime\prime}(\rho)\text{ and }\\
        &3\left( 2p^\prime(\rho)+ \rho p^{\prime\prime}(\rho) \right)\ .\qedhere
      \end{align*}
        \end{proof}
      \end{prop}
The next condition we will replace is condition \ref{item:B_b_positive}. 

\begin{prop}
  \label{prop:B3_simpler}
  Condition \ref{item:B_b_positive} is fulfilled if either of these conditions is satisfied:
  \begin{enumerate}
  \item[\normalfont{C3(b)}] \begin{enumerate}[label=(\roman*)]
    \item $\lim_{\rho \to 0} c(\rho) = 0$ and $2p^\prime(\rho)-\rho p^{\prime\prime}(\rho) \geq 0$ for all $\rho>0$.
    \item $0<\lim_{\rho \to 0} c(\rho) < \infty$.
    \item There is a $\eta \in (0,1)$ such that $\lim_{\rho \to 0}\rho^\eta c(\rho)$ exists and $0 <\lim_{\rho \to 0}\rho^\eta c(\rho) <\infty$.
    \end{enumerate}
\end{enumerate}
  \begin{proof}\ 
    \begin{enumerate}[label=(\roman*)]
    \item Here we have
      \begin{equation*}
        \begin{aligned}
          \lim_{\rho \to 0}\int_{\rho}^{\rho_l}  \frac{c(s)}{s}\dif s - c(\rho) - c(\rho_l) & = \lim_{\rho \to 0}\int_{\rho}^{\rho_l}  \left( \frac{c(s)}{s} - c^\prime(s)\right)\dif s - 2c(\rho)\\
          & = \lim_{\rho \to 0}\int_{\rho}^{\rho_l}  \left( \frac{c(s)}{s} - c^\prime(s)\right)\dif s \\
          &= \lim_{\rho \to 0}\int^{\rho_l}_\rho  \frac{1}{2s\sqrt{p^\prime(s)}}(2p^\prime(s)-sp^{\prime\prime}(s))\dif s\\
          &\geq 0.
      \end{aligned}
    \end{equation*}
  \item In this case the integral is unbounded near zero, but  $c(\rho)+c(\rho_l)$ is finite.
  \item In this case there is $a >0$ such that $\rho^\eta c(\rho) = a\left( 1+r(\rho) \right)$ with $r(\rho) \to 0$.  Then, for every $m \in \mathbb{N}$ there is $\rho_m>0$ such that for all $\rho< \rho_m$ there holds $c(\rho) \geq a\left( 1-\frac{1}{m} \right)\rho^{-\eta}$.  Choose now $m$ so great that
    \begin{equation*}
      \left(1-\frac{1}{m}\right) \frac{1}{\eta}>1
    \end{equation*}
    and then $\rho_{m,0}$ so small that for all $\rho<\rho_{m,0}$ there holds
    \begin{equation*}
      \left(1-\frac{1}{m}\right) \frac{1}{\eta}-r(\rho) = 1+ \theta
    \end{equation*}
    with $\theta>0$.  Define also
    \begin{equation*}
      C_m = \int_{\rho_m}^{\rho_l}  \frac{c(s)}{s}\dif s.
    \end{equation*}
    Then we find for $\rho<\rho_{m,0}$
      \begin{align*}
        \int_{\rho}^{\rho_l}  &\frac{c(s)}{s}\dif s - c(\rho) - c(\rho_l) = \int_{\rho}^{\rho_m}  \frac{c(s)}{s}\dif s + \int_{\rho_m}^{\rho_l}  \frac{c(s)}{s}\dif s - c(\rho) - c(\rho_l)\\
        & \geq a \left(1-\frac{1}{m}\right) \int_\rho^{\rho_m}s^{-\eta-1}\dif s + C_m -c(\rho) -c(\rho_l)\\
        & = a \left( \left(1-\frac{1}{m}\right) \frac{1}{\eta}-1 -r(\rho)\right)\rho^{-\eta}  -a \left(1-\frac{1}{m}\right) \frac{1}{\eta}\rho_m^{-\eta} + C_m  -c(\rho_l)\\
        &> a \theta\rho^{-\eta}  -a \left(1-\frac{1}{m}\right) \frac{1}{\eta}\rho_m^{-\eta} + C_m  -c(\rho_l)\\
        & \overset{\rho \to 0}{\to} +\infty>0. \qedhere
    \end{align*}
      \end{enumerate}
  \end{proof}
\end{prop}
\begin{prop}
  \label{prop:B2_simpler}
  All conditions only depend on differences in pressures, that is, if $c \in \mathbb{R}$ is a constant then if a pressure function $p$ satisfies our conditions, so does $p+c$.  This means that if $p_\infty<\infty$ in condition \ref{item:B_negative}, we can instead choose $p_\infty = 0$, yielding the clearer condition
  \begin{equation*}
    p^\prime(\rho)\leq -\frac{p(\rho)}{\rho}\ .
  \end{equation*}
\end{prop}
Having all we need, we now sum up our findings in the following Proposition~\ref{prop:sufficient}:
\stepcounter{thm}
\begin{conprop}[Proposition \thethm]
  \label{prop:sufficient}
  The Riemann problem for arbitrary sub-sonic left- and right-hand states is well-posed if the pressure function $p$ satisfies the following conditions:
\begin{enumerate}[label=\normalfont{(\theconcounter\arabic*)},ref=\normalfont{\theconcounter}\arabic*]
  \item \label{item:C_concave} both inequalities of Proposition~\ref{prop:B1_simpler}:
    \begin{align}
      2p^\prime(\rho)+ \rho p^{\prime\prime}(\rho)&\geq 0,\\
      6p^\prime(\rho) + 6\rho p^{\prime\prime}(\rho)+\rho^2p^{\prime\prime\prime}(\rho)&\geq 0,
    \end{align}
    \item \label{item:C_negative} one of the conditions according to Proposition~\ref{prop:B2_simpler}:
\begin{enumerate}[label=(\alph*), ref=\normalfont{\theconcounter}\ensuremath{\arabic{enumi}}(\alph{enumii})]
      \item $p \to \infty$ for $\rho \to \infty$ or
      \item $p^\prime(\rho)\leq -\frac{p(\rho)}{\rho}$ for all $\rho>0$,
      \end{enumerate}
    \item \label{item:C_positive} according to Proposition~\ref{prop:B3_simpler}:
    \begin{enumerate}[label=(\alph*), ref=C\arabic{enumi}(\alph{enumii})]
    \item There is $p_0 >0$ such that $p(\rho) = -\frac{p_0}{\rho} + o\left( \frac{1}{\rho} \right)$ for $\rho \to 0$ or
    \item one of the following conditions is fulfilled: \label{item:C_b_positive}
      \begin{enumerate}[label=(\roman*)]
      \item $\lim_{\rho \to 0} c(\rho) = 0$ and $2p^\prime(\rho)-\rho p^{\prime\prime}(\rho) \geq 0$ for all $\rho>0$,
      \item $0<\lim_{\rho \to 0} c(\rho) < \infty$,
      \item there is a $\eta \in (0,1)$ such that $\lim_{\rho \to 0}\rho^\eta c(\rho)$ exists and $0 <\lim_{\rho \to 0}\rho^\eta c(\rho) <\infty$.
      \end{enumerate}
    \end{enumerate}
    \end{enumerate}
    These conditions are positively linear in $p$ and hence all pressure functions satisfying them define a convex cone.  In addition, because the integral is monotone and linear, we can also integrate over pressure functions to find new pressure functions. 
  \end{conprop}
  With Proposition \ref{prop:sufficient} we have a useful list of conditions which can be checked easily for any given candidate for a pressure function.
  The proposition now makes it easy to prove our central result Proposition \ref{prop:general_gamma} and we will do so now.
    \begin{proof}[Proof of Proposition \ref{prop:general_gamma}]\ 
      \begin{itemize}
      \item For the first part of the claim we note that for $\abs{\delta}\leq 2$ and $\alpha>0$ the generalized $\gamma$-law fulfills the conditions in \ref{prop:sufficient}.  Note that at $\delta = 0$ the conditions in \ref{item:C_positive} switch and at $\delta =-1$ the conditions in \ref{item:C_negative} switch.
      \item For the second part of the claim we observe:
        For $\delta>2$, choose $-c(\rho_l)<\frac{q_l}{\rho_l}<-\frac{2}{\delta}c(\rho_l)$, which is obviously sub-sonic.  Then $L_l$ is strictly negative for $\rho>0$, as is easily computed.  A similar range for $\frac{q_r}{\rho_r}$ shows the same for $-L_r$.  For $\delta<-2$, choose $\sqrt{\frac{-1}{\delta+1}}c(\rho_l)< \frac{q_l}{\rho_l}< c(\rho_l)$ and $L_l$ is strictly positive for $\rho>0$ and similarly for $-L_r$.\qedhere
  \end{itemize}
  \end{proof}

\begin{example}
  As just proved, the functions
  \begin{equation*}
    \begin{aligned}
    p(\rho)&=\rho^3\\
    p(\rho)&=\rho\\
    p(\rho)&=-\rho^{-1}\\
    p(\rho)&=-\rho^{-\frac{1}{2}}\\    
    p(\rho)&=\ln(\rho)\\
\end{aligned}
\end{equation*}
are all valid pressure functions. Even a and more exotic
\begin{equation*}
  p(\rho) = \frac{\rho^3-\rho}{\ln(\rho)} = \int_1^3\rho^\gamma \dif \gamma
\end{equation*}
might be possible.
\end{example}

\subsection{Extension to junctions}
\label{sec:extension-junctions}
Our findings guarantee well-posedness of Riemann problems with sub-sonic initial conditions under certain conditions on the pressure function.
We now examine the coupling conditions \eqref{eq:pressure_coupling} and \eqref{eq:flux_conservation} and do so by considering the generalized Riemann problem at the junction in accordance with \cite{coupling_conditions_isothermal,banda_herty_klar,doi:10.1137/060665841}.
Consider a junction with incoming pipes indexed by $e \in {I}$ and outgoing pipes indexed by $f \in {O}$. 
At the junction-facing end of each pipe there are initial states $U_i = (\rho_i,q_i)$ for each $i \in I$ and $i \in O$, respectively.
To make things tractable, we restrict the solution of the junction Riemann problem to be of this form: In each pipe $i \in I\cup O$, there appears exactly one new state $V_i$ next to the junction such that the $V_i$ satisfy the junction conditions and are connected to their respective $U_i$ by a shock or a rarefaction wave. On ingoing pipes these must be $1$-waves, on outgoing pipes these must be $2$-waves. A sketch of this is shown in figure \ref{fig:junction}.

\begin{figure}[h]
  \centering
  \begin{subfigure}{0.49\textwidth}
    \centering
    \includegraphics[width=0.49\textwidth]{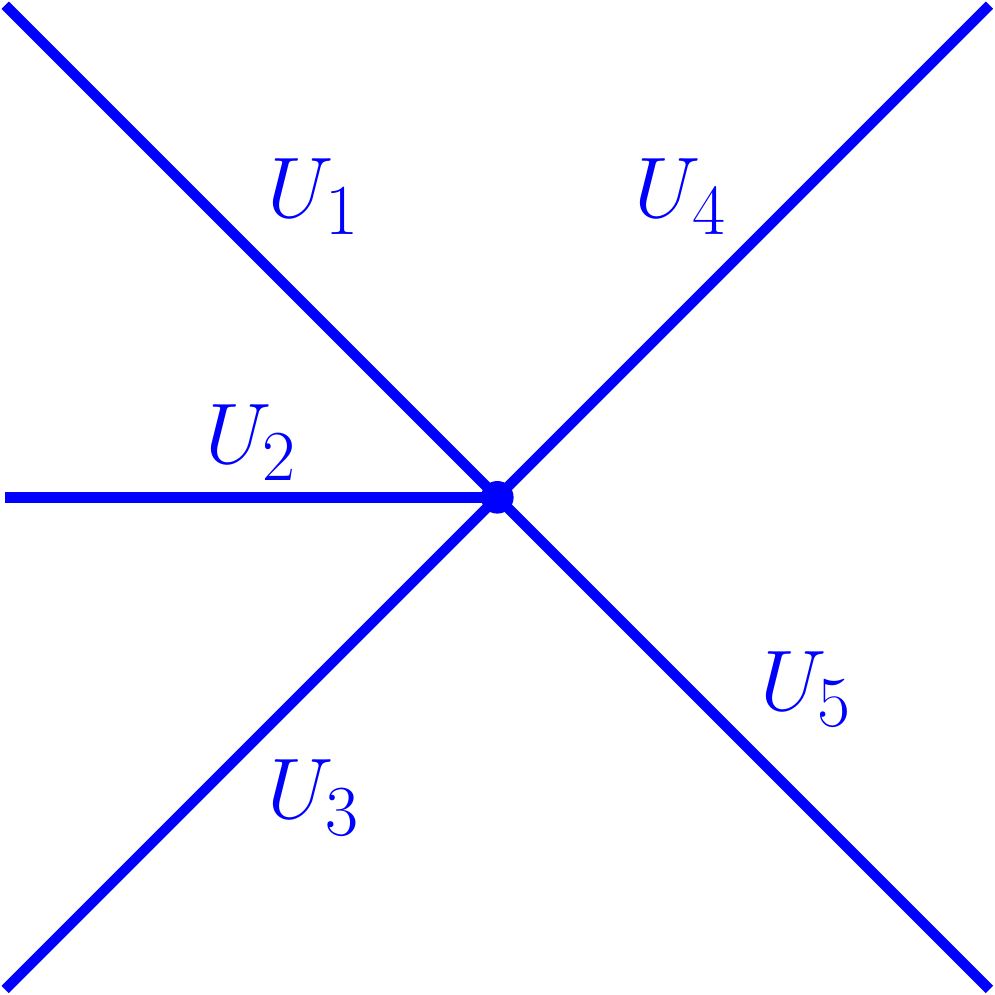}
    \caption{Junction at $t = 0$. }
  \end{subfigure}
  \begin{subfigure}{0.49\textwidth}
    \centering
    \includegraphics[width=0.49\textwidth]{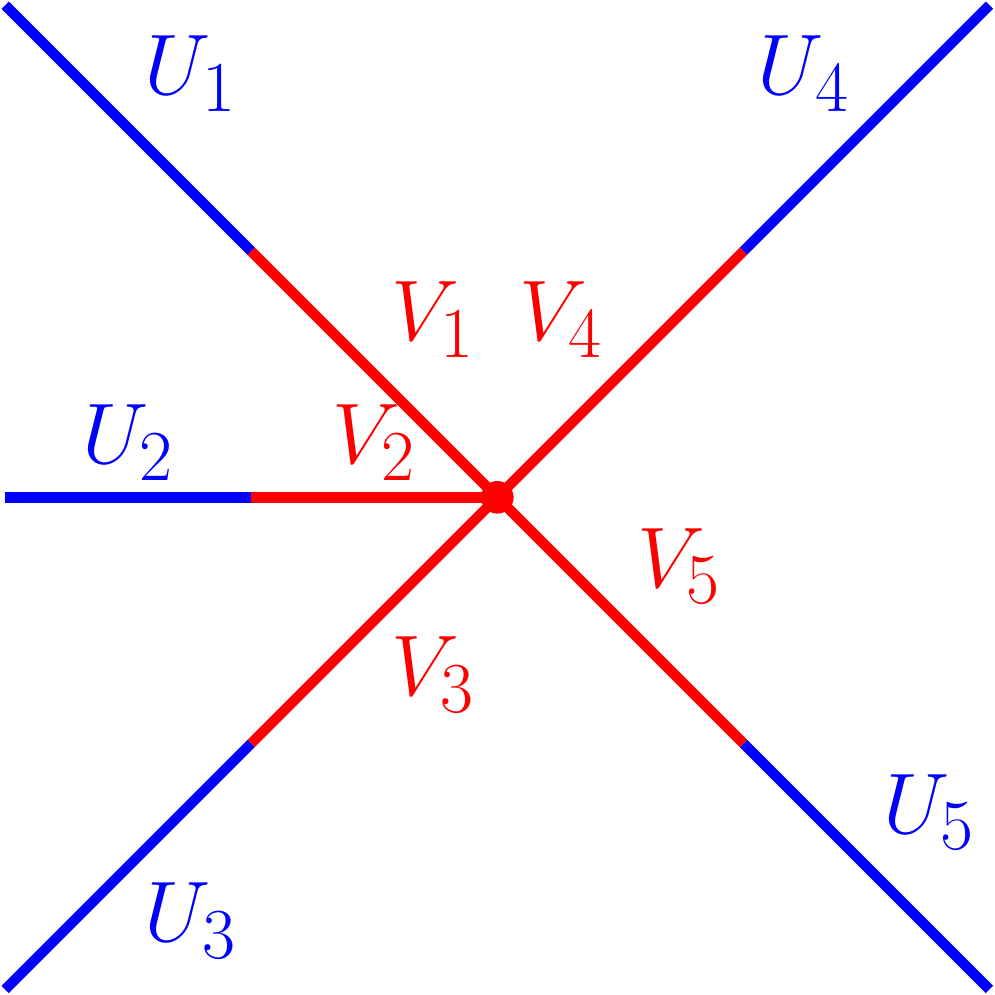}
    \caption{Junction after some time. }
  \end{subfigure}
  \caption{A junction with initial states. $i = 1,2,3$ are incoming pipes, $i=4,5$ are outgoing.}
  \label{fig:junction}
\end{figure}

This is not a great restriction as this is the only solution structure found in pipelines with low Mach number, yet for different solutions, see \cite{MartinHertyMueller2017}.
For these new states to appear at all, the wave speed between $V_i$ and $U_i$ must be negative on ingoing pipes and positive on outgoing ones.
To keep things simple, we examine a single ingoing pipe $i$ with initial condition $U_i$ and new state $V_i= (\rho,L_l(\rho;U_i))$.
\begin{prop}
  Let Proposition \ref{prop:proposition_b} and hence \ref{prop:Lax_curves} be fulfilled for the pressure.
  Let $\rho_{i,\text{min}}$ be such that $L_l^\prime(\rho_{i,\text{min}};U_i) = 0$, if it exists, otherwise let $\rho_{i,\text{min}}=0$.
  The wave speed of the wave between $U_i$ and $V_i$ is negative if and only if $\rho>\rho_{i,\text{min}}$.
  \begin{proof}
    One easily sees that $\lambda_1(\rho,L_l(\rho;U_i)) = L_l^\prime(\rho;U_i)$.
    Also using the definitions \eqref{eq:L_l} and \eqref{eq:derivative_L_l} and the concavity of $L_l$, one can see that $\rho_{i,\text{min}}$ is always unique.
    Because $L_l$ is concave, $L_l^\prime$ is decreasing, so $\lambda_1(\rho,L_l(\rho;U_i))<0$ if and only if $\rho>\rho_{i,\text{min}}$.
    \begin{itemize}
    \item[$\Leftarrow$:] Let $\rho>\rho_{i,\text{min}}$. Then $\lambda_1(V_i)<0$.  If $V_i$ is in the rarefaction part of $L_l(\rho;U_i)$ the wave speed is negative.
      If $V_i$ is in the shock part of $L_l(\rho;U_i)$, we compute:
      \begin{equation*}
        s^\prime(\rho)= \left(\frac{L_l(\rho;U_i)-L(\rho_i;U_i)}{\rho-\rho_i}  \right)^\prime= \frac{L_l^\prime(\rho;U_i)(\rho-\rho_i)-(L_l(\rho;U_i)-L_l(\rho_i;U_i))}{(\rho-\rho_i)^2}    \leq 0\ ,
      \end{equation*}
      since $L_l$ is concave.  So
      \begin{equation*}
        s(\rho)\leq s(\rho_i) = \lambda_1(\rho_i)<0
      \end{equation*}
      because $U_i$ is sub-sonic.
    \item[$\Rightarrow$:] Let $\rho\leq\rho_{i,\text{min}}$. Because $U_i$ is sub-sonic, we have $\rho\leq\rho_{i,\text{min}}<\rho_i$ and are dealing with a rarefaction wave.
      But then the wave speed is just given by $\lambda_1(\rho,L_l(\rho;U_i)$ which is non-negative. \qedhere
    \end{itemize}
  \end{proof}
\end{prop}
For outgoing pipes one defines $\rho_{i,\text{min}}\geq 0$ such that $L_r^\prime(\rho_{i,\text{min}};U_i) = 0$, if possible, or $\rho_{i,\text{min}} = 0$ otherwise and obtains in the same way
\begin{prop}
  Let Proposition \ref{prop:proposition_b} be fulfilled for the pressure.  On an outgoing pipe the wave speed between $U_i$ and $V_i= (\rho,L_r(\rho;U_i))$ is positive if and only if $\rho>\rho_{i,\text{min}}$.
\end{prop}
Coming back to junctions we define the \emph{junction minimal density} as 
\begin{equation}
  \label{eq:junction_minimal_density}
  \rho_\text{min} = \max_{i \in I\cup O}\rho_{i,\text{min}}\ .
\end{equation}
A solution $(V_i)_{i \in I\cup O}$ to the junction Riemann problem is admissible if and only if the density $\rho$ at the junction fulfills $\rho>\rho_\text{min}$.
If it is not admissible, then in at least one pipe there is a super-sonic gas flow.

Note that a usual Riemann problem with sub-sonic initial conditions can be treated as a junction with one ingoing and one outgoing pipe.
In this case only one new state $V$ is created and the admissibility criterion guarantees that $\lambda_1(V)<0<\lambda_2(V)$.
For the sake of completeness, we classify the solutions to the usual Riemann problem by wave types in Table \ref{tab:wave_types}.

\begin{table}[h]
  
  \caption{(r)arefaction waves and (s)hocks for different values of the density $\rho$\label{tab:wave_types}.}
  \begin{subtable}[l]{0.49\textwidth}
    \caption{ wave types for $\rho_l\leq \rho_r$.}
  \label{tab:1}
  \begin{tabular}[h]{l@{\hskip 3pt}r@{\hskip 3pt}c@{\hskip 3pt}l@{\hskip 3pt}l@{\hskip 20pt}cc} 
    \toprule
    \multicolumn{5}{c}{value of $\rho$} & left wave & right wave\\
    \midrule
    &&$\rho$&$\leq$& $\rho_\text{min}$&\multicolumn{2}{c}{invalid} \\ 
    $\rho_\text{min}$&$<$&$\rho$&$\leq$&$ \rho_l$& r &r\\ 
    $\rho_l$&$ \leq $&$\rho$&$\leq $&$\rho_r$& s &r\\ 
    $\rho_r $&$\leq$&$ \rho$&&& s &s\\
    \bottomrule
  \end{tabular}
\end{subtable}
\begin{subtable}[r]{0.49\textwidth}
  \caption{wave types for $\rho_r\leq \rho_l$.}
  \label{tab:2}
  \begin{tabular}[h]{l@{\hskip 3pt}r@{\hskip 3pt}c@{\hskip 3pt}l@{\hskip 3pt}l@{\hskip 20pt}cc} 
    \toprule
    \multicolumn{5}{c}{value of $\rho$} & left wave & right wave\\
    \midrule
                                        &&$\rho$&$\leq$& $\rho_\text{min}$&\multicolumn{2}{c}{invalid} \\ 
    $\rho_\text{min}$&$<$&$\rho$&$\leq$&$ \rho_r$& r &r\\ 
    $\rho_r$&$ \leq $&$\rho$&$\leq $&$\rho_l$& r &s\\ 
    $\rho_l $&$\leq$&$ \rho$&&& s &s\\
    \bottomrule
  \end{tabular}
\end{subtable}
\end{table}

\subsection{Additional constraints for consistency}
\label{sec:addit-constr-cons}
We have found a complete list of conditions for a usual Riemann problem and a junction Riemann problem to be well-posed.
For the usual Riemann problem this is enough, as just noted.
Yet our conditions do not guarantee that the solution states at the junction are sub-sonic, as only one of each eigenvalue is restricted.
Since sub-sonic states are necessary at the junction, we introduce an additional condition by hand. 
We again focus on one incoming pipe $i$ with initial condition $U_i$ and new state $V_i= (\rho,L_l(\rho;U_i))$.
We define
\begin{equation}
  \label{eq:rho_max_pipe}
  \rho_{i,\text{max}} = \min(\set{\rho>0|\lambda_2(\rho,L_l(\rho;U_i))<0}\cup \set{\infty})
\end{equation}
and further
\begin{equation}
  \label{eq:rho_max_junction}
    \rho_\text{max} = \min_{i \in I\cup O}\rho_{i,\text{max}}.
\end{equation}
Although there is no knowledge about $\lambda_2(V_i)$ as it is not conveniently given by the derivative of a Lax curve, we still have
$\lambda_2(V_i)=\lambda_1(V_i)+2c(\rho)>\lambda_1(V_i)$, so there holds at least $\rho_{i,\text{max}}>\rho_{i,\text{min}}$ in every pipe.  But there is no guarantee that $\rho_\text{max}>\rho_\text{min}$.




\section{Coupling to power grids}
\label{sec:coupling-power-grids}
In this section, we focus on the coupling of the gas network to power grids.
Following the ideas in \cite{aleksey_michael,Zlotnik2016}, the coupling is done extending the already existing coupling conservation of fluxes \eqref{eq:flux_conservation} by a gas power plant.
To model the power grid, we use the well-known powerflow equations \cite{grainger2016power}, which describe real and reactive power nodewise.
We have to investigate again under which assumptions this kind of coupling is well-posed.

\subsection{Power flow model}
\label{sec:power-flow-model}
Power grids are usually modeled as a graph, whose nodes (buses) are indexed by $k$ and whose edges (transmission lines) are indexed by the indices of the nodes connected by the edge.
For each node $k \in V_{\PG}$ of the power grid (PG) there are four (time-dependent) state variables, namely real power $P_k(t)$, reactive power $Q_k(t)$, voltage magnitude $\abs{V_k(t)}$ and voltage angle $\phi_k(t)$.
We consider three different types of nodes: load/PQ buses ($P$ und $Q$ are given), generater/PV buses ($P$ and $\abs{V}$ are given) and a single slack bus ($\abs{V}$ and $\phi$ are given).
Thus we always have two unknowns per node $k \in V_\PG$.
Accordingly, we apply the powerflow equations for real and reactive power for each node:
\begin{subequations}
\begin{align}
 P_k &= \sum\limits_{j\in V_{\PG}} \vert V_k \vert \, \vert V_j \vert \, \big( G_{kj} \cos(\va_k - \va_j) + B_{kj} \sin(\va_k - \va_j) \big),\label{eq:P}\\
 Q_k &= \sum\limits_{j\in V_{\PG}} \vert V_k \vert \, \vert V_j \vert \, \big( G_{kj} \sin(\va_k - \va_j) - B_{kj} \cos(\va_k - \va_j) \big),\label{eq:Q}
\end{align}
\end{subequations}
where $G_{kj}$ is the real part of the entry $y_{kj}$ in the bus admittance matrix and $B_{kj}$ is the imaginary part.
For $k \ne j$ we consider $G_{kj}$ and $B_{kj}$ as properties of the arc/transmission line connecting nodes $k$ and $j$.
$G_{kk}$ and $B_{kk}$ are considered as node properties.
\begin{figure}[htb]
  \centering
  \includegraphics{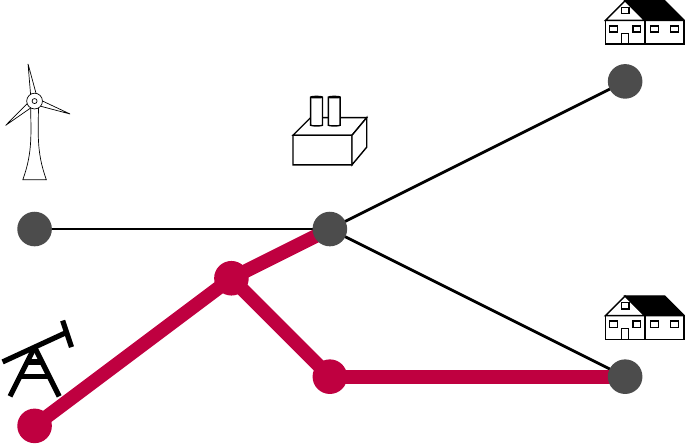}
\caption{Sketch of a gas network (purple) coupled to a power grid (black).\label{fig:coupling}}
\end{figure}

\subsection{Gas-to-power coupling}
\label{sec:coupling-1}
We focus on a junction with one incoming pipeline, one outgoing pipeline and an outlet that draws a set amount of flow $\fuel$, which is converted to electric power.
The coupling condition is taken from \cite{aleksey_michael} and reads
\begin{subequations}\label{eq:31}
  \begin{align}
    p_{\text{in}} &= p_{\text{out}}\label{eq:13}\\
    q_{\text{in}} &= q_{\text{out}} +\fuel\ ,\label{eq:11}
  \end{align}
  \end{subequations}
where the outflow $\fuel$ is non-negative.  
The fuel is converted to gas via the following \emph{heat rate} formula, taken from \cite{Zlotnik2016},
\begin{equation}
\fuel(P) = a_0 + a_1 P + a_2 P^2, \label{eqn:generatorconsumption} 
\end{equation} 
where $P$ is the real power and $a_0$, $a_1$, $a_2$ are constants.    

To find what $\fuel$ are allowed, we examine another generalized Riemann problem at the junction. We start with two constant states on the ingoing and outgoing pipes and solve the Riemann problem where the flow $q$ jumps by an amount $\fuel$ at the junction.  We again demand the waves to leave the junction, giving rise to a picure like that in Figure \ref{fig:solution},
\begin{figure}[h]
  \centering
\includegraphics{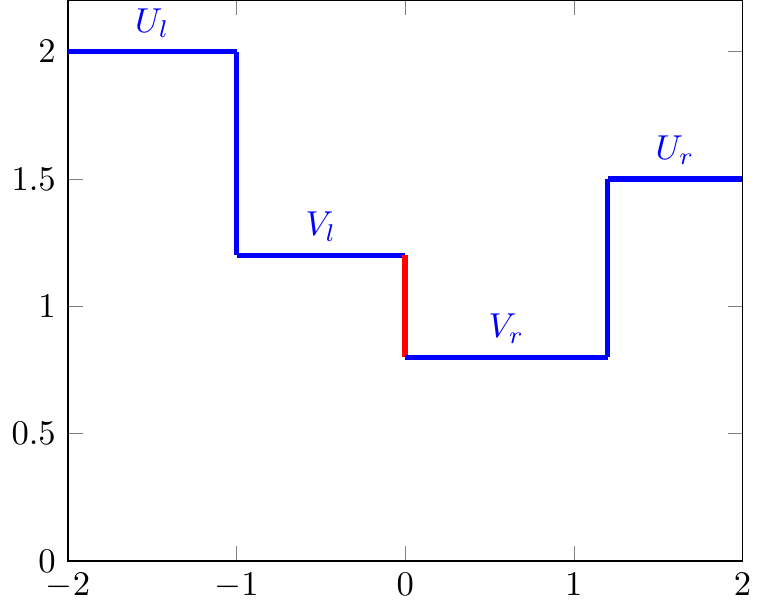}
  \caption{Solution after a short time.}
  \label{fig:solution}
\end{figure}
where next to the junction in red two new states $V_l$, $V_r$ appear that fulfill the coupling conditions \eqref{eq:31}.  
The analysis of this setting is similar to a three-way junction but with the flow on one outgoing pipe fixed to $\fuel$.  We must now solve the equation
\begin{equation}\label{eq:7}
  L_l(\rho) - L_r(\rho) = \fuel\ ,
\end{equation}
which is similar to equation \eqref{eq:10} of a usual (two-way) Riemann problem but with a non-zero right-hand side.  
We even can forego the additional consistency constraints of section \ref{sec:addit-constr-cons}:
One half of the eigenvalues is on the right side of zero due to the $\rho_\text{min}$-criterion, which we need here again to make the waves have the right direction.
\begin{equation*}
  \lambda_1(V_l)\leq 0 \quad \text{and} \quad  \lambda_2(V_r)\geq 0\ ,
\end{equation*}
Instead of invoking $\rho_\text{max}$, we note that $\fuel\geq 0$ and compute for the remaining two eigenvalues
\begin{equation*}
    \lambda_2(V_l)  = \frac{L_l(\rho)}{\rho}+c(\rho) 
     = \frac{L_r(\rho)+\fuel}{\rho}+c(\rho)
     = \lambda_2(V_r)+\frac{\fuel}{\rho}
     \geq \lambda_2(V_r)
     \geq 0\ ,
\end{equation*}
and
\begin{equation*}
    \lambda_1(V_r)  = \frac{L_r(\rho)}{\rho}-c(\rho) 
      = \frac{L_l(\rho)-\fuel}{\rho}-c(\rho) 
      \lambda_1(V_l)-\frac{\fuel}{\rho}
     \leq \lambda_1(V_l)
     \leq 0\ .
\end{equation*}

\begin{remark}
  In our setting this is sufficient.  If we also considered power-to-gas plants, we would need to also treat negative $e$.  In this case, we would have to adhere to the $\rho_\text{max}$-criterion \eqref{eq:rho_max_junction} again.
\end{remark}

\subsubsection*{Solution structure for different outflows $\fuel$}
\label{sec:what-e-produces}
For the usual Riemann problem we had a unique non-zero solution due to our findings in Section \ref{sec:analysisOutflowBounds}.  For $\fuel>0$ we now have two solutions to $L_l(\rho) - L_r(\rho) = \fuel$, one of which is not admissible as it lies to the left of the maximum and hence has non-negative derivative (which means it would be super-sonic).
As $L_l-L_r$ is decreasing in the admissible regime, greater $\fuel$ result in smaller $\rho$.  One of the two solution structures is given in table \ref{tab:solstruct}.

\begin{table}[h]
  \centering
  \caption{Solution structure for $\rho_l\leq \rho_r$.\label{tab:solstruct}}
  \begin{tabular}[h]{l@{\hskip 3pt}r@{\hskip 3pt}c@{\hskip 3pt}l@{\hskip 3pt}l@{\hskip 30pt}cc} 
    \toprule
    \multicolumn{5}{c}{value of $\fuel$} & left wave & right wave\\
    \midrule
    &&$\fuel$&$ \geq $&$(L_l-L_r)(\rho_\text{min})$ & \multicolumn{2}{c}{invalid} \\
    $(L_l-L_r)(\rho_\text{min})$&$>$&$ \fuel$&$\geq$&$ (L_l-L_r)(\rho_l)$ &r &r\\
    $(L_l-L_r)(\rho_l)$&$\geq $&$\fuel$&$\geq $&$(L_l-L_r)(\rho_r)$ & s& r\\
    $(L_l-L_r)(\rho_r)$&$\geq$&$ \fuel$ &&& s& s\\
    \bottomrule
  \end{tabular}
    \end{table}

\section{Numerical results}
\label{sec:numerics}

Within the following numerical examples, we consider two different discretization schemes. The first is a third-order CWENO scheme (CWENO3) with suitable boundary treatment \cite{Kolb2014,NaumannKolbSemplice2018}, which relies on a local Lax-Friedrichs flux function for the inner discretization points of each pipe and handles coupling points by explicitly solving equation \eqref{eq:7}.

The second scheme is an implicit box scheme (IBOX)~\cite{KolbLangBales2010}, suitable for sub-sonic flows. For a general system of balance laws
\begin{equation}
  \label{eq:4}
  U_t + F(U)_x = G(U),
\end{equation}
the considered scheme reads
\begin{multline}
  \label{eq:6}
  \frac{U^{n+1}_{j-1} + U^{n+1}_{j}}{2} = \frac{U^{n}_{j-1} + U^{n}_{j}}{2} - \frac{\Delta t}{\Delta x}\left( F(U^{n+1}_j)- F(U^{n+1}_{j-1}) \right) + \Delta t \frac{G(U^{n+1}_j)+G(U^{n+1}_{j-1})}{2}.
\end{multline}
Here, $\Delta t$ and $\Delta x$ are the temporal and spatial mesh size, respectively, and the numerical approximation is thought in the following sense:
\begin{equation}
   U_j^n \approx U(x,t) \quad \text{for} \quad x\in \big[ (j-\tfrac{1}{2})\Delta x , (j+\tfrac{1}{2})\Delta x \big) , \ t \in \big[ n\Delta t , (n+1)\Delta t \big).
\end{equation}

The implicit box scheme has to obey an inverse CFL condition~\cite{KolbLangBales2010}, which is beneficial for problems with large characteristic speeds whereas the solution is merely quasi-stationary. This is usually the case for daily operation tasks in gas networks and therefore motivates the choice of this scheme for the real-world scenario below.

The first test example in section~\ref{sec:validation} is supposed to demonstrate the different cases revealed in the analysis above (section~\ref{sec:coupling-1}). Further, since the applied implicit box scheme does not explicitly make use of any Riemann solver, this scenario is also considered as a numerical validation of its applicability, where the CWENO3 scheme with Riemann solver at the junction serves as reference.

Within the second example (section~\ref{sec:numericalResultsDifferentPressure}) we briefly demonstrate the differences resulting from various pressure functions, which are all covered by our theoretical results above.

The third example (section~\ref{sec:numericalResultsGasPower}) considers a more complex scenario: An increasing power demand within the power grid leads to an increasing fuel demand of a gas-to-power generator and further to a significant pressure drop in the gas network. 

\subsection{Validation}\label{sec:validation}
We consider the isentropic Euler equations with pressure law $p(\rho) = \kappa \rho^\gamma$ and parameters $\kappa = 1.0$, $\gamma = 1.4$, and a Riemann problem with left state $U_l = \begin{pmatrix}4.0\\ 1.0\end{pmatrix}$ and right state $U_r = \begin{pmatrix}3.0\\ -1.0\end{pmatrix}$. Further we assume a gas demand $\fuel$ at the coupling point of the two states (here $x=0$). Then, from table~\ref{tab:solstruct}, we get the following solution structure:
\begin{itemize}
 \item s-s solution for $\fuel \le 0.57877$,
 \item r-s solution for $0.57877 \le \fuel \le 3.0594$,
 \item r-r solution for $3.0594 \le \fuel$.
\end{itemize}
Further, one can easily compute $\rho_{1,\min} \approx 1.8819$, $\rho_{2,\min} \approx 1.5041$, and therewith $\rho_{\min} \approx 1.8819$ and maximum gas demand $\fuel \approx 4.3892$. We will consider the numerical simulation of the described setting until time $t=0.1$ for $\fuel\in\{0.25, 1.75, 3.25\}$ and the following discretization parameters:
\begin{itemize}
 \item CWENO3: $\Delta t=5\cdot10^{-5}$, $\Delta x=5\cdot10^{-4}$,
 \item IBOX: $\Delta t=5\cdot10^{-4}$, $\Delta x=5\cdot10^{-5}$.
\end{itemize}
The different choices result from the (usual) CFL condition the explicit CWENO3 scheme has to obey, in contrast to the inverse CFL condition of the IBOX scheme.
Figures~\ref{fig:densityMinus} to~\ref{fig:densityPlus} show the computed densities at the final time. Both schemes show the correct solution structure (shock/rarefaction waves), where CWENO3 expectably achieves the sharper resolution.



\begin{figure}[htb]
\centering
\includegraphics{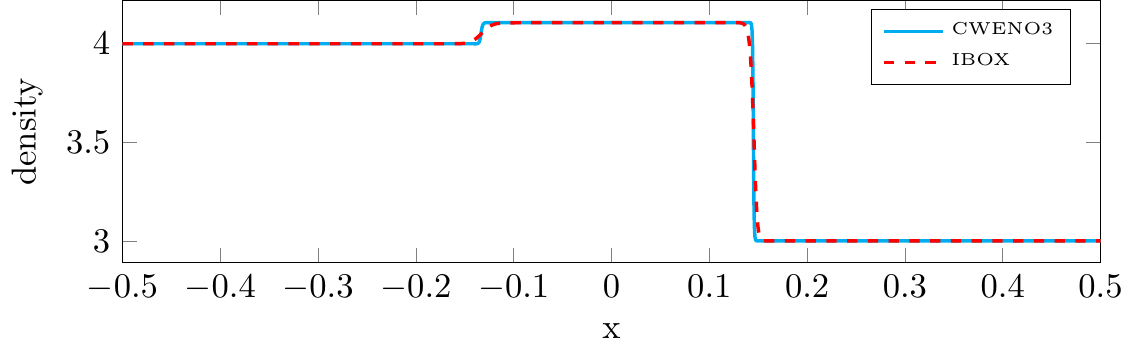}
 \caption{Density profile at $t=0.1$ for $\fuel=0.25$. (s-s solution)}
 \label{fig:densityMinus}
\end{figure}

\begin{figure}[htb]
\centering
  \includegraphics{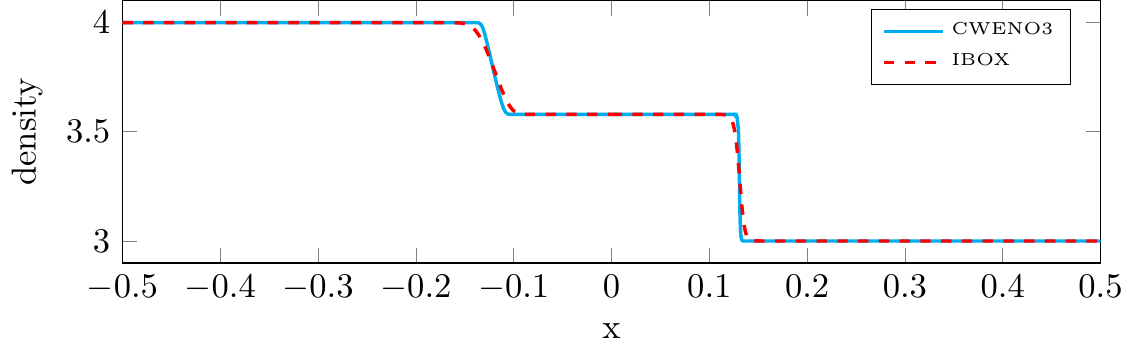}
  \caption{Density profile at $t=0.1$ for $\fuel=1.75$. (r-s solution)}
 \label{fig:density0}
\end{figure}

\begin{figure}[htb]
\centering
\includegraphics{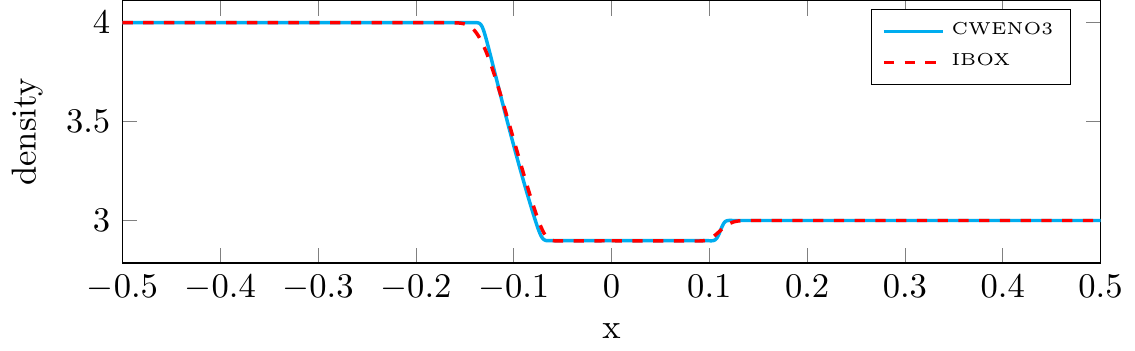}
 \caption{Density profile at $t=0.1$ for $\fuel=3.25$. (r-r solution)}
 \label{fig:densityPlus}
\end{figure}


\subsection{Different pressure laws}\label{sec:numericalResultsDifferentPressure}
In our second test case, we apply various pressure laws, which are all covered by our theoretical study, and are interested in the different dynamics one may observe even on a single pipeline. Therefore, we consider a single pipe with length $l=0.1$ and the following pressure laws:
\begin{itemize}
 \item $p(\rho) = \frac{1}{\gamma} \rho^\gamma$ with $\gamma = 1.4$ (``gamma law''),
 \item $p(\rho) = \frac{-1}{\rho}$ (``inverse'', corresponding to $\gamma = -1$),
 \item $p(\rho) = \ln(\rho)$ (``logarithmic''),
 \item $p(\rho) = \frac{1}{10} \sum\limits_{i=1}^{10} \frac{\rho^{1 + i/5}}{1 + i/5}$ (``sum of gamma laws'').
\end{itemize}
Note that all considered pressure functions are scaled in such a way that $p'(\rho=1) = 1$. Initially, we have $\rho = 1$ and $q = 0$ in the whole pipe. Further, we fix $\rho = 1$ on the left-hand boundary, whereas $q$ at the right-hand boundary linearly increases from $0$ to $0.2$ until time $t=0.1$ and stays constant afterwards until the final time $t=0.5$. We approximate the solution to this problem by CWENO3 with discretization parameters $\Delta t = 5 \cdot 10^{-4}$ and $\Delta x = 10^{-3}$. The variety of the resulting dynamics is demonstrated in figures~\ref{fig:pressure025} and~\ref{fig:pressure050}, which show the density in the pipeline at times $t=0.25$ and $t=0.5$, respectively.

\begin{figure}[htb]
  \centering
  \includegraphics{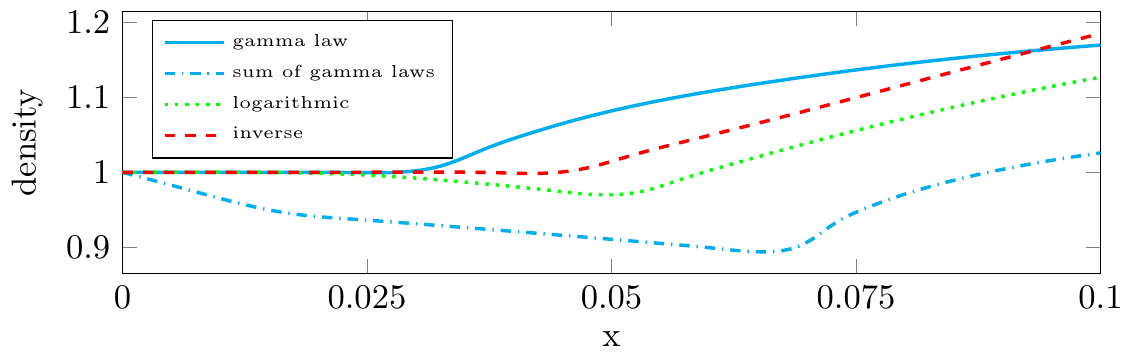}
  \caption{Simulation result at time $t=0.25$ for different pressure laws.}
 \label{fig:pressure025}
\end{figure}

\begin{figure}[htb]
  \centering
  \includegraphics{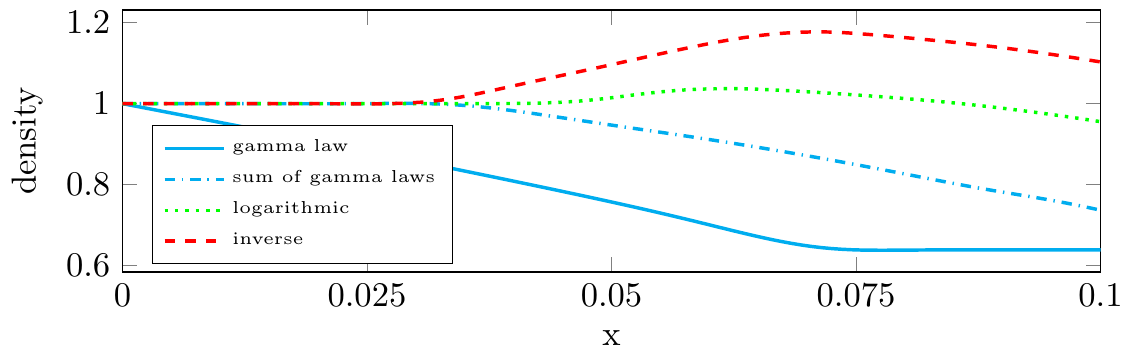}
  \caption{Simulation result at time $t=0.5$ for different pressure laws.}
 \label{fig:pressure050}
\end{figure}


\subsection{Coupled gas and power grid}\label{sec:numericalResultsGasPower}
We consider the network(s) depicted in figure~\ref{fig:gas2power}, containing a power grid from the example ``case9'' of the MATPOWER Matlab programming suite~\cite{MATPOWER} and a small part of the GasLib-40 network~\cite{Humpola_et_al:2015}, extended by a compressor station in front of node S17 and a gas-to-power generator between S4 and N1, providing the necessary power at the latter node. The parameters of the gas network are given in table~\ref{tab:gasNetwork}, the parameters of the power grid in table~\ref{tab:powerGrid}.

\begin{figure}[h!]
 \centering
\includegraphics{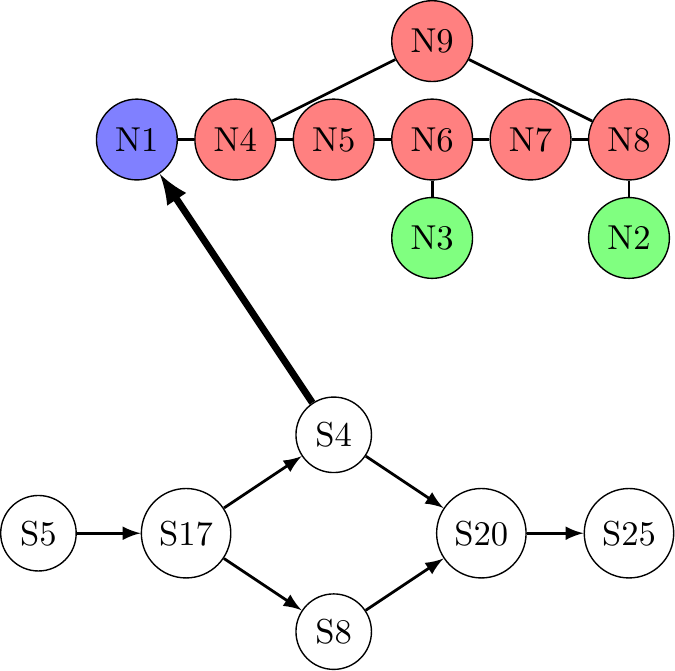}

 \caption{Gas network connected to a power grid. Red nodes are \textcolor{red!50}{PQ/demand nodes}, green nodes are generators (\textcolor{green!50}{PV nodes}) and the blue node is the \textcolor{blue!50}{slack bus} (also a generator, with gas consumption of the form $\fuel(P)=a_0 + a_1 P + a_2 P^2$).}
 \label{fig:gas2power}
\end{figure}

\begin{table}[hbt]
  \centering
  \caption{Parameters of the gas network.\label{tab:gasNetwork}}
    \begin{tabular}{@{}llllll}
    \toprule
    Pipe & From & To & Length [km] & Diameter [mm] & Roughness [mm]\\
    \midrule
    P10 & S4  & S20 & 20.322 & 600 & 0.05\\
    P20 & S5  & S17 & 20.635 & 600 & 0.05\\
    P21 & S17 & S4  & 10.586 & 600 & 0.05\\
    P22 & S17 & S8  & 10.452 & 600 & 0.05\\
    P24 & S8  & S20 & 19.303 & 600 & 0.05\\
    P25 & S20 & S25 & 66.037 & 600 & 0.05\\
    P99 & S4  & S8  & \phantom{1}5.000 & 600 & 0.05\\
    \bottomrule
  \end{tabular}
\end{table}

\begin{table}[hbt]
   \caption{Parameters of the power grid\label{tab:powerGrid}.}
  \begin{subtable}{0.4\textwidth}
    \centering
    \subcaption{Busses}
    \begin{tabular}{@{}lll}
   \toprule
  Node & $G$ & $B$ \\
  \midrule
  N1 & 0.0000 & -17.3611 \\
  N2 & 0.0000 & -16.0000 \\
  N3 & 0.0000 & -17.0648 \\
  N4 & 3.3074 & -39.3089 \\
  N5 & 3.2242 & -15.8409 \\
  N6 & 2.4371 & -32.1539 \\
  N7 & 2.7722 & -23.3032 \\
  N8 & 2.8047 & -35.4456 \\
  N9 & 2.5528 & -17.3382 \\
  \bottomrule
    \end{tabular}
  \end{subtable}
  \centering
      \begin{subtable}{0.4\textwidth}
     \subcaption{Transmission lines}
     \begin{tabular}{@{}lllll}
    \toprule
  Edge & From & To & $G$ & $B$ \\
  \midrule
  TL14 & N1 & N4 & \phantom{-}0.0000 & 17.3611 \\
  TL45 & N4 & N5 & -1.9422 & 10.5107 \\
  TL56 & N5 & N6 & -1.2820 & \phantom{1}5.5882 \\
  TL36 & N3 & N6 & \phantom{-}0.0000 & 17.0648 \\
  TL67 & N6 & N7 & -1.1551 & \phantom{1}9.7843 \\
  TL78 & N7 & N8 & -1.6171 & 13.6980 \\
  TL82 & N8 & N2 & \phantom{-}0.0000 & 16.0000 \\
  TL89 & N8 & N9 & -1.1876 & \phantom{1}5.9751 \\
  TL94 & N9 & N4 & -1.3652 & 11.6041 \\
  \bottomrule
     \end{tabular}
   \end{subtable}
\end{table}

Within each pipe $e\in E_\Gas$ of the gas network (with length $l_e$, diameter $d_e$, cross section $A_e$, roughness $k_e$) we consider the isothermal Euler equations
%
%
with acoustic speed $c=340 \frac{\text{m}}{\text{s}}$.
The source term in the momentum equation is given by
\[
 S(\rho,q) = -\frac{\lambda(q)}{2d_e} \frac{q \vert q\vert}{\rho}
\]
with friction factor $\lambda(q)$, which is determined by the Prandtl-Colebrook formula:
\begin{equation*}
 \frac{1}{\sqrt{\lambda}} = -2\log_{10} \left( \frac{2.51}{\Re(q) \sqrt{\lambda}} + \frac{k_e}{3.71 d_e} \right)
\end{equation*}
with Reynolds number
\begin{equation*}
 \Re(q) = \frac{d_e}{\eta} q
\end{equation*}
and dynamic viscosity $\eta = 10^{-5} \frac{\text{kg}}{\text{ms}}$.


Initially, the gas network is in a stationary state: The pressure at S5 is fixed at 60bar, the ouflow at S25 is $q=100 \frac{\text{m}^3}{\text{s}} \cdot \frac{\rho_0}{A_e}$ with $\rho_0=0.785 \frac{\text{kg}}{\text{m}^3}$, and there is an additional gas consumption at S4 resulting from the gas to power transformation ($a_0=2$, $a_1=5$, $a_2=10$) due to the power demand at the slack bus N1. The initial (stationary) state of the power grid is determined by boundary conditions given in table \ref{tab:powerGridInitial}.

\begin{table}[htb]
 \centering
 \caption{Initial boundary conditions of the power grid.}
    \begin{tabular}{@{}lllll}
    \toprule
  Node & $P$ & $Q$ & $\vert V\vert$ & $\va$\\
  \midrule
  N1 & -    & -   & 1 & 0 \\
  N2 & 163  & -   & 1 & - \\
  N3 & 85   & -   & 1 & - \\
  N4 & 0    & 0   & - & - \\
  N5 & -90  & -30 & - & - \\
  N6 & 0    & 0   & - & - \\
  N7 & -100 & -35 & - & - \\
  N8 & 0    & 0   & - & - \\
  N9 & -125 & -50 & - & - \\
  \bottomrule
 \end{tabular}
 \label{tab:powerGridInitial}
\end{table}

In the course of the simulation, the power (and reactive power) demand at N5 is(/are) linearly increased between $t=1$ hour and $t=1.5$ hours from $0.9$ \pu\ to $1.8$ \pu\ (reactive power from $0.3$ \pu\ to $0.6$ \pu) - see also figure~\ref{fig:powerN5slack} (left). Accordingly, the power demand at the slack bus N1 increases (see figure~\ref{fig:powerN5slack} (right)) and therewith the gas consumption at S4, which also results in an increase of the inflow at S5 (see figure~\ref{fig:inflowS5}). Due to the increased flow values, the pressure in the gas network decreases (see figure~\ref{fig:pressureS20S25} for the pressure at the nodes S20 and S25).

\begin{figure}[h!]
  \centering
  \begin{subfigure}{0.49\textwidth}
    \centering
    \includegraphics[width=\textwidth,height=0.2\textheight]{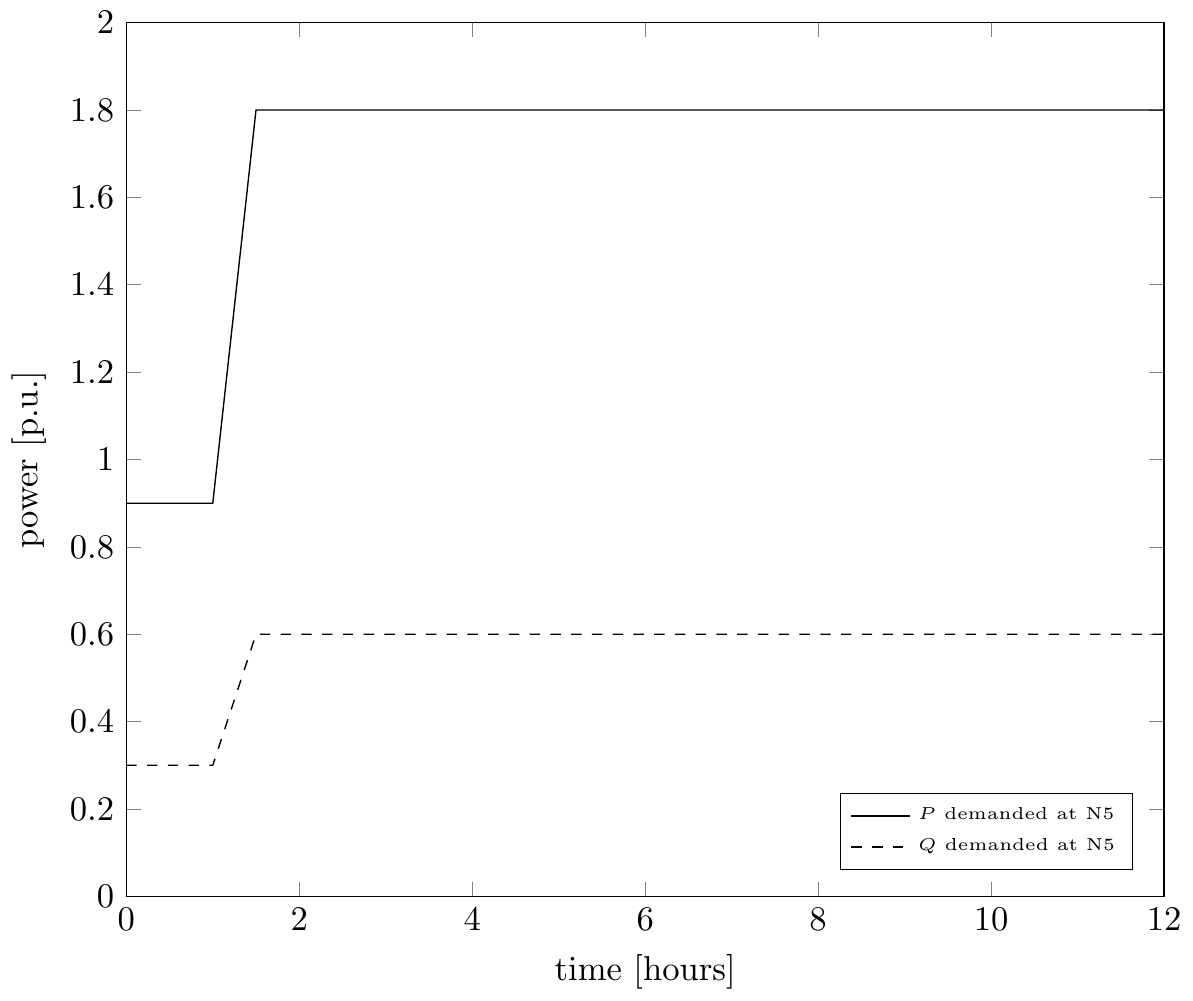}
  \end{subfigure}
  \begin{subfigure}{0.49\textwidth}
    \centering
    \includegraphics[width=\textwidth,height=0.2\textheight]{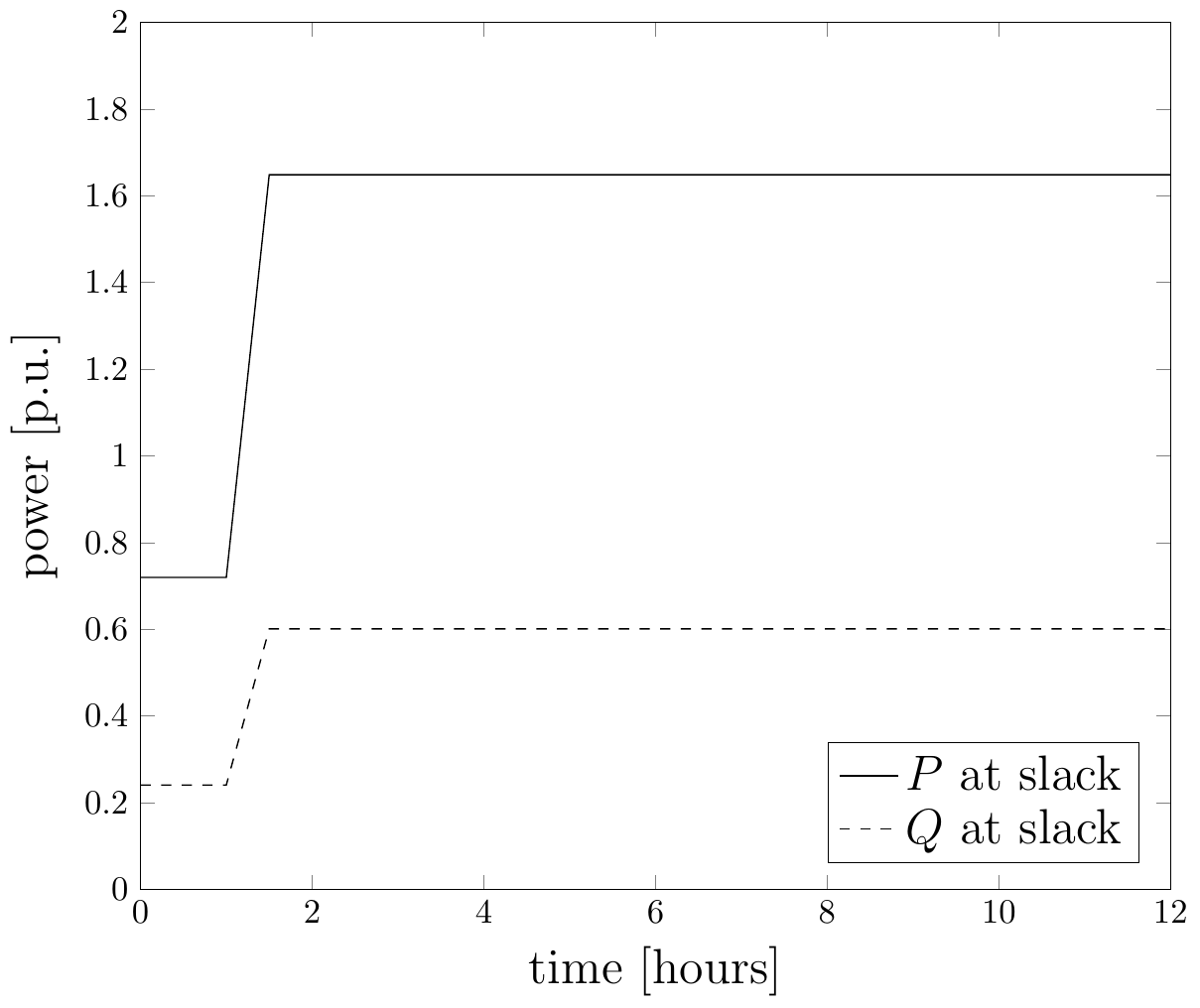}
  \end{subfigure}
  \caption{Power and reactive power at demand node N5 (left) and the slack bus (right).}
  \label{fig:powerN5slack}
\end{figure}

\begin{figure}[htb]
  \centering
  \includegraphics{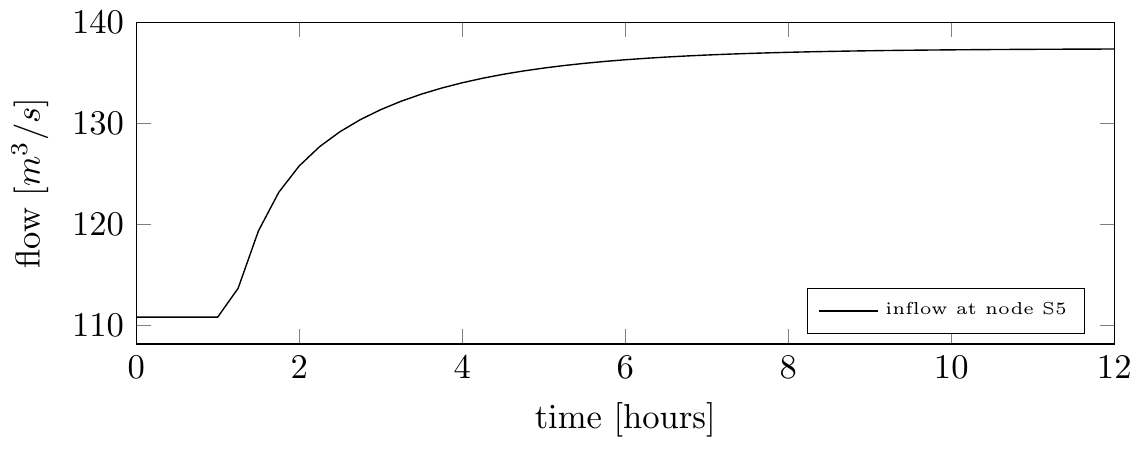}
 \caption{Inflow at node S5.}
 \label{fig:inflowS5}
\end{figure}

\begin{figure}[htb]
  \centering
  \includegraphics{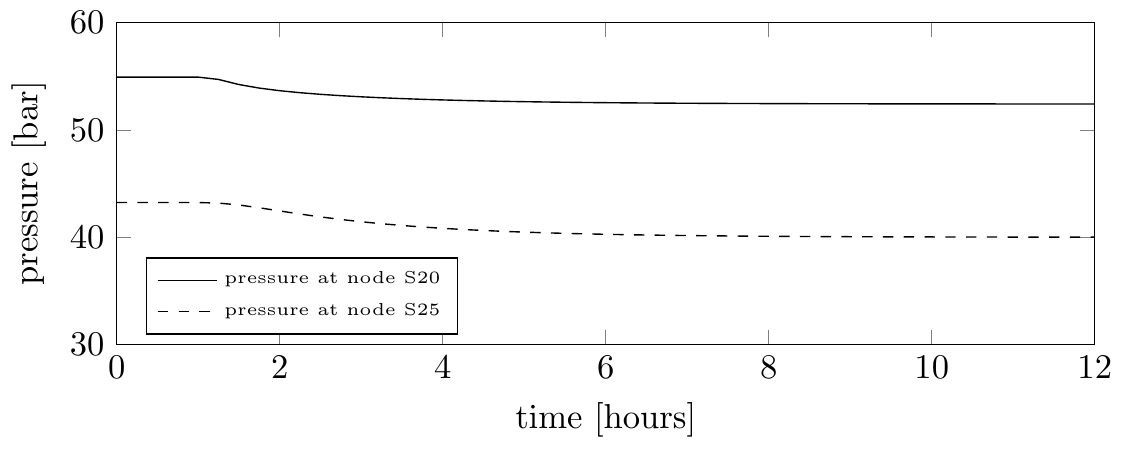}
 \caption{Pressure at nodes S20 and S25.}
 \label{fig:pressureS20S25}
\end{figure}




\section{Conclusion and future work}
We have presented a coupled model for gas and power
allowing for a mathematically well-defined transition from gas to power.
The framework involves also the consideration of non-standard pressure
functions. Various simulation results show the properties of the presented approach.

Future work includes the investigation of optimal control problems for the coupled model
as for example the control of compressor stations~\cite{Gugat2018} or the inclusion of uncertain customer demands~\cite{2018arXiv181005480G}.
This will immediately lead to the class of nonlinear (stochastic) optimization problems and tailored solution techniques.

\section*{Acknowledgments}
The authors gratefully thank the BMBF project ENets (05M18VMA) for the financial support.


\bibliographystyle{siam}
\bibliography{bibliography}

\appendix
\section{Proof of Lemma \ref{sec:condition-lemma}}
\label{sec:proofs-lemm-refs}

The following two technical results are the key ingredients to prove Lemma \ref{sec:condition-lemma}.
\begin{lemma}
  \label{lemma:1}
  Let $g \in C^1(\mathbb{R}^+, \mathbb{R}^+)$be a non-negative function, $g\geq0$, and let $G$ be given by $G(\rho) = \int_\rho^{\rho_l}g(s)\dif s$.  Then There holds
  \begin{enumerate}
  \item If $\rho^2g(\rho) \overset{\rho \to 0}{\to} 0$, then $\rho G(\rho) \overset{\rho \to 0}{\to} 0$.
  \item If $\rho G(\rho) \overset{\rho \to 0}{\to} 0$, then $\liminf_{\rho \to 0}\rho^2g(\rho) =0$.
  \end{enumerate}
\begin{proof}\ 
    \begin{enumerate}
    \item By assumption $\rho^2g(\rho) \overset{\rho \to 0}{\to} 0$. For $m \in
      \mathbb{N}$ choose $\rho_m >0$ such that $\rho^2g(\rho)\leq
      \frac{1}{m\rho^2}$ for $\rho<\rho_m$. Now choose $\rho_{m,0}<\rho_m$ so
      small that
    \begin{equation*}
      \rho_{m,0} \left(\int_{\rho_m}^{\rho_l}g(s)\dif s-\frac{1}{m\rho_m}\right)\leq \frac{1}{m}
    \end{equation*}
    Then, for $\rho <\rho_{m,0}$, there holds
    \begin{equation*}
      \begin{aligned}
        \rho\int_{\rho}^{\rho_l}g(s)\dif s &\leq \rho\int_{\rho_m}^{\rho_l}g(s)\dif s + \frac{1}{m}\rho\int_{\rho}^{\rho_m}\frac{1}{s^2}\dif s
         = \rho \left( \int_{\rho_m}^{\rho_l}g(s)\dif s-\frac{1}{m\rho_m}\right) +\frac{1}{m}\leq \frac{2}{m}
      \end{aligned}
    \end{equation*}
    for small enough $\rho$.  Therefore $\lim_{\rho\to 0}\rho \int_{\rho}^{\rho_l}g(s)\dif s \leq \frac{2}{m}$ for all $m \in \mathbb{N}$. As $g\geq 0$, we also have $0\leq \lim_{\rho\to 0}\rho \int_{\rho}^{\rho_l}g(s)\dif s$.  Summarizing, we get $\lim_{\rho\to 0}\rho G(\rho) = \lim_{\rho\to 0}\rho \int_{\rho}^{\rho_l}g(s)\dif s  = 0$.
\item     We prove by contradiction:  Assume there are $\rho_0>0$ and $a >0$ such that $\rho^2g(\rho) \geq a$ for $\rho < \rho_0$.  Then $g(\rho)\geq \frac{a}{2}\frac{1}{\rho^2}$ for such $\rho$.  Therefore
      \begin{equation*}
        \begin{aligned}
          \rho \int_\rho^{\rho_l}g(s)\dif s \geq \rho \int_{\rho_0}^{\rho_l}g(s)\dif s + \rho\frac{a}{2}\int_{\rho}^{\rho_0} \frac{1}{s^2} \dif s
          \to  0 + \frac{a}{2} \rho \left[ -\frac{1}{s} \right]_\rho^{\rho_0} \to \frac{a}{2}\ ,
      \end{aligned}
    \end{equation*}
    which contradicts $\rho G(\rho) \overset{\rho \to 0}{\longrightarrow} 0$.\qedhere
\end{enumerate}
\end{proof}
\end{lemma}

\begin{lemma}\label{sec:well-posedn-isentr}
  Let $g \in C^1(\mathbb{R}^+, \mathbb{R}^+)$, $g\geq0$ and $\liminf_{\rho \to 0} \rho^2g(\rho) = 0$.  Let also $\left( \rho^2g(\rho) \right)^\prime \geq 0$.  Then, $\limsup_{\rho \to 0} \rho^2 g(\rho) = \liminf_{\rho \to 0} \rho^2 g(\rho) = 0$.

\begin{proof}
  We prove by contradiction.  Let $\limsup_{\rho \to 0} \rho^2 g(\rho) > \liminf \rho^2 g(\rho)$.  Then it is easily seen that $\liminf_{\rho \to 0} \left( \rho^2g(\rho) \right)^\prime = -\infty< 0$ resulting in a contradiction.
\end{proof}
\end{lemma}


\end{document}